\documentclass{LMCS}

\def\dOi{11(1:11)2015}
\lmcsheading%
{\dOi}
{1--25}
{}
{}
{Nov.~\phantom09, 2013}
{Mar.~25, 2015}
{}

\ACMCCS{[{\bf Mathematics of computing}]: Continuous
  mathematics---Topology---Point-set topology}

\subjclass{F.1.1 Models of Computation, F.4.1 Mathematical Logic,
  G.1.2 Approximation, G.m}

\usepackage{amssymb}           
\usepackage{hyperref}
\usepackage{prettyref}         
\newcommand{\pref}{\prettyref} 	
\newrefformat{prn}{Proposition \ref{#1}}
\newrefformat{pry}{Property \ref{#1}}
\newrefformat{def}{Definition \ref{#1}}

\newtheorem{property}{Property}
\newtheorem*{acknowledgements*}{Acknowledgements}
\newcommand{\N}{\mathbb{N}}
\newcommand{\Z}{\mathbb{Z}}
\newcommand{\R}{\mathbb{R}}
\newcommand{\Q}{\mathbb{Q}}
\newcommand{\Zplus}{\Z^+}
\newcommand{\Qplus}{\Q^+}
\newcommand{\divides}{\mid}
\newcommand{\union}{\cup}
\newcommand{\disjunion}{\dot{\cup}}
\newcommand{\absval}[1]{\lvert #1\rvert}
\newcommand{\onespace}{\quad}
\newcommand{\Th}{\textrm{th}}

\DeclareMathOperator{\diam}{diam}
\DeclareMathOperator{\ord}{ord} 
\DeclareMathOperator{\clvar}{cl}

\DeclareMathOperator{\dom}{dom} 
\DeclareMathOperator{\img}{im} 
\DeclareMathOperator{\id}{id} 

\newcommand{\setconstr}{\divides}
\newcommand{\cl}{\overline}
\newcommand{\compose}{\circ}

\newcommand{\trfns}{R^{(1)}} 
\newcommand{\Bairespc}{\mathbb{B}} 
\newcommand{\pr}{\ensuremath{\pi}}

\newcommand{\emptystring}{\lambda}       
\newcommand{\pwrset}{\mathcal{P}}

\newcommand{\mathwrt}{\text{ w.r.t.~}} 
\newcommand{\mathst}{\text{ s.t.~}} 
\newcommand{\ball}[2]{B(#1;#2)}
\newcommand{\clball}[2]{\bar{B}(#1;#2)}
\newcommand{\ballmetric}[3]{B_{#1}(#2;#3)}
\newcommand{\clballmetric}[3]{\bar{B}_{#1}(#2;#3)}
\newcommand{\nbhd}[2]{N_{#2}(#1)} 
\newcommand{\clnbhd}[2]{\bar{N}_{#2}(#1)} 
\newcommand{\hausdmetr}{d_{\textrm{H}}} 
\newcommand{\bdry}{\partial}

\newcommand{\Cont}{\mathit{C}}  
\newcommand{\leastmu}[2]{\mathbf{\mu}#1\left(#2\right)} 
\newcommand{\ordomega}{\mathit{\omega}} 

\newcommand{\FS}{\textrm{FS}} 

\newcommand{\crsarr}{\rightrightarrows} 
\newcommand{\KO}{\mathcal{KO}} 
\newcommand{\deltac}{\delta_{\textrm{cover}}} 
\newcommand{\deltamc}{\delta_{\textrm{min-cover}}} 
\newcommand{\deltadc}{\delta_{\textrm{disj-cover}}'} 
\newcommand{\deltar}{\delta_{\textrm{range}}} 

\begin{document}
\title[Effective zero-dimensionality]{Effective zero-dimensionality for computable\\ metric spaces\rsuper*}

\author[R.~Kenny]{Robert Kenny}	
\address{School of Mathematics \& Statistics, The University of Western Australia, Perth, Australia}
\email{robert.kenny@uwa.edu.au}  

%



\keywords{zero-dimensional, dimension theory, computable analysis}
\titlecomment{{\lsuper*}Results in \pref{sec:zdsubsets} have been presented at CCA 2013, Nancy, France, Jul 8-10 2013}


\begin{abstract}
  \noindent We begin to study classical dimension theory from the
  computable analysis (TTE) point of view.  For computable metric
  spaces, several effectivisations of zero-dimensionality are shown
  to be equivalent.  The part of this characterisation that concerns
  covering dimension extends to higher dimensions and to closed
  shrinkings of finite open covers.  To deal with zero-dimensional
  subspaces uniformly, four operations (relative to the space and a
  class of subspaces) are defined; these correspond to definitions of
  inductive and covering dimensions and a countable basis condition.
  Finally, an effective retract characterisation of
  zero-dimensionality is proven under an effective compactness
  condition.  In one direction this uses a version of
  the construction of bilocated sets.
\end{abstract}

\maketitle


\section{Introduction}\label{sec:intro}
Various spaces of symbolic dynamics \cite{LindMarcus}, such as $X=A^{\N}$ for a finite alphabet $A$ or the sofic subshifts, are useful examples of zero-dimensional topological spaces, interesting both for dynamics and in connection with computation.  Some similar remarks apply to the spaces of cellular automata $A^{\Z^n}$ and to a lesser extent to general subshifts.
To deal effectively with sets which are zero-dimensional in non-symbolic mathematical contexts, however (such
as in $\R^n$ or the minimal sets of an expansive compact dynamical system \cite[Thm 2.2.44]{AokiHiraide}), it
is desirable to examine possible effective versions of this property.  In the present work, we begin a basic
investigation to consider effective zero-dimensionality both of computable metric spaces and of their closed
subsets, in the framework of computable analysis via representations (see \cite{Weihrauchbook2},
\cite{BrattkaHertlingWeihrauch08}).

To this end, for a topological space $X$, recall that a subset
$B\subseteq X$ is \emph{clopen} if $B$ is open and closed, equivalently if the boundary $\bdry B$ is empty.
For a separable metrizable space $X$, the following conditions are equivalent:
\begin{enumerate}
\item\label{zdii} $(\forall p\in X)(\forall A\in\Pi^0_1(X))\left( p\not\in A\implies
\emptyset\text{ is a partition between $p$ and $A$} \right)$,
\item\label{zdiii} $(\forall A,B\in\Pi^0_1(X))\left( A\cap B=\emptyset\implies
\emptyset\text{ is a partition between $A$ and $B$} \right)$,
\item\label{zdiv} $(\forall \mathcal{U})(\exists \mathcal{V})\left( \mathcal{U}\text{ is an open cover of }X
\implies\mathcal{V}\text{ is a partition into open sets refining $\mathcal{U}$} \right),$
\item\label{zdv}
$(\forall \mathcal{U})(\exists \mathcal{V})\bigl(
\mathcal{U}\text{ is a finite open cover of }X\implies
\mathcal{V}\text{ is a finite partition into open sets}\\ \text{ refining $\mathcal{U}$}
\bigr),$
\item\label{zdvi} there exists a countable basis $\mathcal{B}$ for the topology of $X$ consisting of clopen
sets,
\item\label{zdviii} $(\forall \mathcal{U})(\exists \mathcal{V}) \bigl(
 \mathcal{U}\text{ is a finite open cover}\implies
\mathcal{V}\text{ is an open shrinking of $\mathcal{U}$ by pairwise}\\
\text{ disjoint sets}\bigr),$
\item\label{zdix} $(\forall A\in\Pi^0_1(X))(\exists f\in \Cont(X,X))\left( A\neq\emptyset\implies \img f=A\wedge f|_A=\id_A \right).$
\end{enumerate}
Here, in (\ref{zdii}) and (\ref{zdiii}), $P$ is a \emph{partition} between disjoint $A,B\subseteq X$ if there
exist disjoint open $U,V\subseteq X$ such that $A\subseteq U$, $B\subseteq V$ and $X\setminus P=U\union V$.
In (\ref{zdiv}) and (\ref{zdv}) a partition (of $X$) is a pairwise disjoint family of sets (with union equal to
$X$).
In (\ref{zdviii}) a \emph{shrinking} of a cover $(A_i)_{i\in I}$ of $X$ is a cover $(B_i)_{i\in I}$
satisfying $B_i\subseteq A_i$ for all $i\in I$.
A nonempty space $X$ satisfying (\ref{zdii}) (or any of the equivalent conditions) is \emph{zero-dimensional};
a subset $Y\subseteq X$ is \emph{zero-dimensional} if $Y$ is zero-dimensional in the relative topology
$\mathcal{T}_X|^Y$.

Next, recall that any zero-dimensional separable metrizable $X$ is homeomorphic to a subspace of
the Cantor space $C:=\{0,1\}^{\N}$ \cite[Thm 7.8, p 38]{Kechris}.  For strictly topological questions on
zero-dimensional spaces it is thus possible to consider only subspaces of $C$.
In this paper we will address our questions from the slightly more intrinsic point of view mentioned above, treating zero-dimensionality on a computable metric space $X$ and its subsets.
More specifically, we consider computable versions of the existence statements (\ref{zdii})-(\ref{zdix});
these are certain multi-valued
operations which, stopping short of studying Weihrauch degrees, we require to be computable.  In the case of a
subset $Y\subseteq X$, zero-dimensionality of $Y$ can be stated in several ways using closed or open subsets
of $X$, and these statements also can be viewed as multi-valued operations.  While a systematic treatment is not
given, we present various definitions of operations (corresponding to equivalent forms of
zero-dimensionality) and some of their interrelations.

Thus, in \pref{sec:zdsubsets} three implications are proven between four operations relevant for a general class
$\mathcal{Y}\subseteq\pwrset(X)$ of zero-dimensional or empty subsets of cardinality
$\absval{\mathcal{Y}}\leq 2^{\aleph_0}$; these correspond to (\ref{zdii}), (\ref{zdiii}), (\ref{zdvi}) and,
loosely, to a condition like (\ref{zdiv}) or (\ref{zdviii}).
Further results on the four operations for $\mathcal{Y}=\{Y\in\Pi^0_1(X)\setconstr\dim Y\leq 0\}$
under effective local compactness or similar assumptions
will be discussed elsewhere.  In
\pref{sec:zdspcs} the results of \pref{sec:zdsubsets} are specialised to the case $Y=X$, and a robust notion
of effectively zero-dimensional computable metric space is found to exist.
Some more evidence for the suitability of that definition is provided by \pref{sec:dimsec}, which deals with
covering dimension (essentially extending the conditions (\ref{zdiv}), (\ref{zdv}) and (\ref{zdviii})), though
in an ad hoc way.

We also present, in \pref{sec:repnsec}, an effective version of the decomposition of totally bounded open subsets of zero-dimensional spaces found in \cite[Cor 26.II.1]{Kuratowski}.  This is used (with an effective
compactness assumption) to prove \pref{thm:retracta}, an effectivization of
(\ref{zdix}) above.  Finally, in \pref{sec:retractsec} a converse \pref{prn:retractb} is proven.  This relies
on the existence of so-called bilocated sets from the constructive analysis literature; some computable
analysis versions of these proofs are given in the same section.  Sections \ref{sec:second} and
\ref{sec:covsec} respectively discuss notation and supporting results on general covering
properties of metric spaces (namely, effective versions of the Lindel{\" o}f property, and swelling and shrinking of finite covers).

\section{Notation}\label{sec:second}
By $\langle\cdot\rangle:\N^{*}\to\N$ and $\langle\cdot,\cdot\rangle:\N^2\to\N$ we denote standard tupling
functions, with corresponding coordinate projections $\pr_1,\pr_2:\N\to\N$ in the binary case.  A standard numbering $\nu_{\N^{*}}$ of $\N^{*}$ is also introduced by $\nu_{\N^{*}}\langle w\rangle:=w$ ($w\in\N^{*}$).
Similarly, with $\Bairespc:=\N^{\N}$ we define $\langle\cdot,\cdot\rangle:\Bairespc^2\to\Bairespc$ and
$\langle\cdot,\dots\rangle:\Bairespc^{\N}\to\Bairespc$ by
\begin{equation*}
\langle p^{(0)},p^{(1)}\rangle(2 i+z)=p^{(z)}_i \: \text{ and }\:
\langle p^{(0)},p^{(1)},\dots\rangle(\langle i,j\rangle) = p^{(i)}_j
\end{equation*}
(here $p=p_0 p_1\dots\in\Bairespc$, i.e.~$p_i:=p(i)$ for every $p\in\Bairespc$, $i\in\N$).  Again we write
$\pr_1,\pr_2:\Bairespc\to\Bairespc$ for the coordinate projections in the binary case.  We will also
occasionally consider projections $\pr_1:X\times Y\to X$ and $\pr_2:X\times Y\to Y$ for any cartesian product
$X\times Y$; it will be clear from the context which of the above notions is meant.
Further, in a metric space $X$, we write
\begin{equation*}
\nbhd{A}{\epsilon}:=\bigcup_{x\in A}\ball{x}{\epsilon}\: \text{ and } \:
\clnbhd{A}{\epsilon}:=\bigcup_{x\in A}\clball{x}{\epsilon}
\end{equation*}
for any $A\subseteq X$ and $\epsilon>0$.

In general, we assume familiarity with the framework of computable analysis via representations
\cite{Weihrauchbook2}, \cite{BrattkaHertlingWeihrauch08}.  We will also use some notation for specific
representations from \cite{BrattkaPresser}.
If $(X_i,\delta_i)$ ($1\leq i\leq n$) and $(Y,\delta')$ are represented spaces, similarly to
\cite{Weihrauchbook2}, a $(\delta_1,\dots,\delta_n;\delta')$-\emph{realiser} of an operation
$f:\subseteq X_1\times\dots\times X_n\crsarr Y$ is a map $F:\subseteq\Bairespc^n\to\Bairespc$ such that
\begin{align*}
& F(p^{(1)},\dots,p^{(n)})\in (\delta')^{-1}f(\delta_1(p^{(1)}),\dots,\delta_n(p^{(n)}))
\: \text{ whenever } \\
& (p^{(1)},\dots,p^{(n)})\in\prod_{i=1}^n\dom\delta_i \:\text{ and }\:
(\delta_1(p^{(1)}),\dots,\delta_n(p^{(n)}))\in\dom f.
\end{align*}

However, unless otherwise mentioned, when representations $\delta_1$, $\delta_2$, $\delta'$ are understood a
`realiser' of $f:\subseteq X_1\times X_2\crsarr Y$ will be a map $F:\subseteq\Bairespc\to\Bairespc$, namely a
$([\delta_1,\delta_2];\delta')$-realiser.  This convention has some minor advantages where brevity is concerned.

For a computable metric space $(X,d,\nu)$, in this paper the
\emph{Cauchy representation} $\delta_X:\subseteq\Bairespc\to X$ is defined by
\begin{equation*}
p\in\delta_X^{-1}\{x\}:\iff \lim_{i\to\infty}\nu(p_i)=x\wedge
(\forall i,j\in\N) \, d(\nu(p_i),\nu(p_j))<2^{-\min\{i,j\}}.
\end{equation*}
A representation $\rho$ of $\R$ will be used less often; for definiteness, let it be the Cauchy representation
of $(\R,d,\nu_{\Q})$, where $d(x,y)=\absval{x-y}$.

Let $(X,\mathcal{T})$ be a second countable topological space and let $\alpha,\beta:\N\to\mathcal{T}$ be
numberings of possibly different countable bases.
\begin{defi}\label{def:incldef}
$(\sqsubset) \subseteq \N^2$ is a \emph{formal inclusion} of $\alpha$ with respect to $\beta$ if
\begin{equation*}
(\forall a,b\in\N)(a\sqsubset b\implies \alpha(a)\subseteq\beta(b)).
\end{equation*}
Consider the following axioms, in order of increasing strength.
\begin{enumerate}
\item\label{inclone} $(\forall b)(\forall x\in X)(\exists a)(x\in\beta(b)\implies x\in\alpha(a)\wedge a\sqsubset b)$
\item\label{incltwo} $(\forall b)(\forall x\in X)(\forall U\in\mathcal{T})(\exists a)\left(
 x\in\beta(b)\cap U\implies x\in\alpha(a)\subseteq U\wedge a\sqsubset b \right)$
\item\label{inclthree} $(\forall a,b)(\forall x\in X)(\exists c)(x\in\beta(a)\cap\beta(b)\implies
 x\in\alpha(c)\wedge c\sqsubset a\wedge c\sqsubset b)$
\item\label{inclfive} $(\forall b)(\forall x\in X)(\exists U\in\Sigma^0_1(X))(\forall a)(x\in\alpha(a)\subseteq \beta(b)\cap U\implies a\sqsubset b)$
\end{enumerate}
\end{defi}
In particular, in a computable metric space $(X,d,\nu)$, consider numberings of ideal open and closed balls
\begin{align*}
\alpha & :\N\to\mathcal{B}:=\img\alpha\subseteq\mathcal{T},\langle a,r\rangle\mapsto
\ballmetric{d}{\nu(a)}{\nu_{\Qplus}(r)},\\
\hat{\alpha} & :\N\to\img\hat{\alpha}\subseteq\Pi^0_1(X),\langle a,r\rangle\mapsto
\clballmetric{d}{\nu(a)}{\nu_{\Qplus}(r)}.
\end{align*}
Here $\nu_{\Qplus}$ is a standard total numbering of the positive rationals $\Qplus$ with a
$(\nu_{\Qplus},\id_{\N})$-computable right-inverse $\overline{\cdot}:\Qplus\to\N$.

The relation $\sqsubset$ defined by
\begin{equation*}
\langle a,r\rangle\sqsubset\langle b,q\rangle :\iff d(\nu(a),\nu(b))+\nu_{\Qplus}(r)<\nu_{\Qplus}(q)
\end{equation*}
is a formal inclusion of $\alpha$ with respect to itself; moreover it satisfies
$c\sqsubset d\implies\hat{\alpha}(c)\subseteq\alpha(d)$ and
(\ref{inclfive}).
For the purposes of this paper, we will often call a formal inclusion satisfying property (\ref{inclone}) a
\emph{refined inclusion}.

From any basis numbering $\alpha$ (of a
topological space $X$) we can define a representation
\begin{equation*}
\delta:\Bairespc\to\Sigma^0_1(X),p\mapsto \bigcup\{\alpha(p_i-1)\setconstr i\in\N, p_i\geq 1\}
\end{equation*}
of the hyperspace of open sets in $X$.  For a computable metric space with
$\alpha$ as above, this representation is denoted $\delta_{\Sigma^0_1(X)}$, or $\delta_{\Sigma^0_1}$ if $X$ is
clear from the context.
Correspondingly, we write
\begin{equation*}
\delta_{\Pi^0_1}:\Bairespc\to\Pi^0_1(X),p\mapsto X\setminus\delta_{\Sigma^0_1}(p)
\end{equation*}
for a representation of the hyperspace of closed sets in $X$, and
\begin{equation*}
\delta_{\Delta^0_1}:\subseteq\Bairespc\to\Delta^0_1(X),\langle p,q\rangle\mapsto \delta_{\Sigma^0_1}(p) =
\delta_{\Pi^0_1}(q)
\end{equation*}
(with natural domain) for a representation of the clopen sets in $X$.  When writing
$\Sigma^0_1(X)$, $\Pi^0_1(X)$, $\Delta^0_1(X)$ we always assume these classes are equipped with the
corresponding representations.

For the purposes of this paper we need two more representations of the class $\mathcal{A}(X)$ of closed sets in
$X$ (cf.~\cite{BrattkaPresser}).  Define
$\deltar,\delta_{\textrm{dist}}^{>}:\subseteq\Bairespc\to\mathcal{A}(X)$ by
\begin{multline*}
\langle p^{(0)},\dots\rangle\in\deltar^{-1}\{A\} :\iff
\bigl( A=\emptyset\wedge(\forall i)p^{(i)}=0^{\ordomega} \bigr) \vee
\bigl( A\neq\emptyset\wedge\{p^{(i)}\setconstr i\in\N\}\subseteq P^{-1}\delta_X^{-1}A \wedge \\
(\forall x\in A)(\forall U\in\mathcal{T}_X)(\exists i)(x\in U\implies
(\delta_X\compose P)(p^{(i)})\in U) \bigr),
\end{multline*}
where $P:\subseteq\Bairespc\to\Bairespc$ is defined by
$P(p)_i:=p_i-1$ ($\dom P=\{p\in\Bairespc\setconstr (\forall i)p_i\geq 1\}$),
\begin{equation*}
p\in(\delta_{\textrm{dist}}^{>})^{-1}\{A\} :\iff\textrm{ $\eta_p$ $(\delta_X,\overline{\rho_{<}})$-realises }
d_A:X\to\bar{\R},
\end{equation*}
where
\begin{equation*}
p\in\overline{\rho_{<}}^{-1}\{t\} :\iff
\{n\in\N\setconstr \nu_{\Q}(n)<t\}=\{p_i-1\setconstr i\in\N\wedge p_i\geq 1\}.
\end{equation*}
Here $\eta_p=\eta(p)$ for a certain `canonical' representation $\eta$ of the set
$\mathbf{F}=\{F:\subseteq\Bairespc\to\Bairespc\setconstr
F\text{ continuous with $\mathcal{G}_{\delta}$ domain}\}$.  In particular, $\eta$ satisfies certain versions
of the smn and utm theorems; see \cite[\S 2.3]{Weihrauchbook2} for precise (and rather general) statements.
$\deltar$ and $\delta_{\textrm{dist}}^{>}$ will be used in Sections \ref{sec:repnsec} and \ref{sec:retractsec}.

Next, for any represented set $(X,\delta)$, consider the set $X^{*}$ of finite-length words over the alphabet
$X$.  A representation of $X^{*}$ is defined by
\begin{equation*}
\delta^{*}:\subseteq\Bairespc\to X^{*},n.\langle p^{(0)},p^{(1)},\dots\rangle\mapsto\begin{cases}
\emptystring, &\mbox{\textrm{ if $n=0$}}\\
\delta(p^{(0)})\dots\delta(p^{(n-1)}), &\mbox{\textrm{ if $n\geq 1$}}\end{cases}\quad (n\in\N),
\end{equation*}
where $\emptystring$ is the empty word.
In Sections \ref{sec:covsec}, \ref{sec:zdspcs}, \ref{sec:dimsec} and
\ref{sec:retractsec} we will use $\delta^{*}$ for various representations $\delta$ of hyperspaces of a fixed
computable metric space $X$.
If $(I,\nu)$ is a numbered set, a representation $\delta_{\nu}:\subseteq\Bairespc\to I$ is defined by
\begin{equation*}
\dom\delta_{\nu}=\{p\in\Bairespc\setconstr p_0\in\dom\nu\}\onespace\text{and}\onespace\delta_{\nu}(p)=\nu(p_0).
\end{equation*}
Consider now the set $E(X)$ of finite subsets of $X$.
For a numbered set $(I,\nu)$ one can define a standard numbering $\FS(\nu)$ of $E(I)$ following
\cite[Defns 2.2.2, 2.2.14(5)]{Weihrauchbook}: first, define a total numbering $e$ of $E(\N)$ by $e=\psi^{-1}$
for the bijection $\psi:E(\N)\to\N,A\mapsto \sum_{i\in A}2^i$.  Then define
\begin{equation*}
\FS(\nu):\subseteq\N\to E(I),k\mapsto \{\nu(i)\setconstr i\in e(k)\}\onespace \text{ where }\onespace
\dom\FS(\nu)=\{k\setconstr e(k)\subseteq\dom\nu\}.
\end{equation*}
The next lemma verifies equivalence of two representations arising from these definitions.
\begin{lem}
For any numbered set $(I,\nu)$,
$\delta_{\FS(\nu)}\equiv\delta_{E(I)}$, where
\begin{equation*}
\begin{split}
p\in \delta_{E(I)}^{-1}\{S\}:\iff
(\exists k)(\forall i)\left( (i<k\implies p_i\in 1+\dom\nu)\wedge(i\geq k\implies p_i=0)\right)\\
\wedge\{\nu(p_i-1)\setconstr i<k\}=S.
\end{split}
\end{equation*}
\end{lem}
\begin{proof}
\textbf{$\delta_{\FS(\nu)}\leq\delta_{E(I)}$}: we use $F:\subseteq\Bairespc\to\Bairespc,a.0^{\ordomega}\mapsto
w.0^{\ordomega}$ where $\absval{w}=\# e(a)$ (the number of nonzero bits in the binary representation of $a$)
and $w_i:=j+1$ if $j$ is the $i^{\Th}$ smallest member of $e(a)$.\\
\textbf{$\delta_{E(I)}\leq \delta_{\FS(\nu)}$}: we use
$F:\subseteq\Bairespc\to\Bairespc,p\mapsto a.0^{\ordomega}$ where
$k:=\leastmu{i}{p_i=0}$ and $a:=\sum\{2^j\setconstr j\in\N\wedge (\exists i<k) p_i=j+1\}$.
\end{proof}

\section{Covering properties}\label{sec:covsec}
For any represented spaces $(X,\delta)$, $(Y,\delta^{\prime})$, denote the set of $(\delta,\delta^{\prime})$-continuous total maps $f:X\to Y$ by $\Cont_{\textrm{s}}(\delta,\delta^{\prime})$.
\begin{lem}\label{lem:ctslindelof}
For computable metric spaces $(X,d,\nu)$, $(Z,d',\nu')$ and
Cauchy representation $\delta_Z$ of $Z$, the computable dense sequence $z_i:=\nu'(i)$ ($i\in\N$) satisfies
\begin{equation*}
\textstyle\bigcup_{i\in\N}u(z_i) = \bigcup_{z\in Z}u(z)
\end{equation*}
for any $u\in\Cont_{\textrm{s}}(\delta_Z,\delta_{\Sigma^0_1(X)})$.
In particular,
\begin{align*}
L' &:\Cont_{\textrm{s}}(\delta_Z,\delta_{\Sigma^0_1(X)})\to\Sigma^0_1(X)^{\N},u\mapsto (u(z_i))_{i\in\N},\\
\union &:\Cont_{\textrm{s}}(\delta_Z,\delta_{\Sigma^0_1(X)})\to\Sigma^0_1(X),u\mapsto
\textstyle\bigcup_{z\in Z}u(z)
\end{align*}
are resp.~$([\delta_Z\to\delta_{\Sigma^0_1(X)}],\delta_{\Sigma^0_1(X)}^{\ordomega})$- and
$([\delta_Z\to\delta_{\Sigma^0_1(X)}],\delta_{\Sigma^0_1(X)})$-computable.
\end{lem}
\pref{lem:ctslindelof} plays a similar role to the Lindel{\"o}f property of separable metric spaces, albeit only
for representation-continuous indexed covers.  The operation of continuous intersection for closed subsets,
dual to $\union$, has been considered in \cite{BrattkaGherardiBorelCpxty}.
\begin{proof}
Take
\begin{equation*}
A:=\{w\in\N^{*}\setconstr w.\Bairespc\cap\dom\delta_Z\neq\emptyset\}
=\{w\in\N^{*}\setconstr (\forall i,j<\absval{w})d'(\nu'(w_i),\nu'(w_j))<2^{-\min\{i,j\}}\}.
\end{equation*}
We let $z_i:=\nu'(i)=\nu'(w_{\absval{w}-1})$ for $\emptystring\neq w\sqsubset i^{\ordomega}$.
Now consider $u\in\Cont_{\textrm{s}}(\delta_Z,\delta_{\Sigma^0_1(X)})$ and a
continuous realiser $F:\subseteq\Bairespc\to\Bairespc$ of $u$.  For any $x\in X$ and $z\in Z$ such that $u(z)\ni x$
and $q\in\delta_Z^{-1}\{z\}$, it holds that
\begin{equation*}
u(z)=(\delta_{\Sigma^0_1(X)}\compose F)(q) = \bigcup\{\alpha(F(q)_n-1)\setconstr n\in\N, F(q)_n\geq 1\};
\end{equation*}
we suppose $x\in\alpha(a)$ where $a+1=F(q)_n$.  Since $F$ is continuous, there exists $w\sqsubset q$ such that any
$r\in w.\Bairespc\cap\dom\delta_Z$ satisfies $F(r)_n=a+1$ and hence $(u\compose\delta_Z)(r)\ni x$.  In
particular this applies
to $r=w.w_{\absval{w}-1}^{\ordomega}\in\delta_Z^{-1}\{\nu'(w_{\absval{w}-1})\}=\delta_Z^{-1}\{z_i\}$ for
$i=w_{\absval{w}-1}$.
\end{proof}

We
continue this section with some results around
shrinkings and swellings of covers; as in the
classical case these are useful to give equivalent definitions of bounds on covering dimension.
Following \cite{Engelking}, these constructions depend on
Urysohn's lemma; we specifically are interested in the
effective form from
\cite{WeihrauchTietzeUry}.
\begin{thm}(Weihrauch \cite[Thm 15]{WeihrauchTietzeUry})\label{thm:urylem}
In a computable metric space $X$, define
\begin{equation*}
U:\subseteq\Pi^0_1(X)^2\crsarr \Cont(X,\R),(A,B)\mapsto \{f\setconstr \img f\subseteq [0,1]\wedge f^{-1}\{0\}=A\wedge f^{-1}\{1\}=B\}
\end{equation*}
($\dom U=\{(A,B)\setconstr A\cap B=\emptyset\}$).
Then $U$ is $([\delta_{\Pi^0_1},\delta_{\Pi^0_1}],[\delta_X\to\rho])$-computable.
\end{thm}

\begin{defi}
For any family $\mathcal{A}=(A_i)_{i\in I}\subseteq\pwrset(X)$, a \emph{swelling} of $\mathcal{A}$ is a family
$(B_i)_{i\in I}$ satisfying $(\forall i)(A_i\subseteq B_i)$ and
\begin{equation}\label{eq:nerve}
\bigcap_{j<m}B_{w_j}=\emptyset \iff \bigcap_{j<m}A_{w_j}=\emptyset
\end{equation}
for any $m\geq 1$, $w\in I^m$.
\end{defi}
Classically, any finite collection of closed subsets has an open swelling
and this construction can
be effectivized given suitable data on the emptiness or nonemptiness of intersections in (\ref{eq:nerve}).  Dually, this result allows
($\delta_{\Sigma^0_1}^{*}$-~and) subcover information for a finite open cover to be used to produce closed or open shrinkings computably.
For the present paper, working with such information (coding it appropriately in representations for covers) is unnecessarily complicated;
we instead consider two partial effectivisations of the proof of \cite[Thm 7.1.4]{Engelking}.
For any indexed family $(A_i)_{i\in I}\subseteq\pwrset(X)$, the \emph{order} of the family,
$\ord(A_i)_{i\in I}$, is here defined as the least $n$ such that $\bigcap_{j\leq n}A_{i_j}$ is empty whenever
$i_0,\dots,i_n$ are distinct elements of $I$
(this definition varies slightly from that in \cite{Engelking}).

\begin{lem}\label{lem:swellbddord}
Let $X$ be a computable metric space.  For any $N\in\N$, the operation
$S_{+,N}:\subseteq\Pi^0_1(X)^{*}\crsarr\Sigma^0_1(X)^{*}$ defined by
$\dom S_{+,N}=\{(F_i)_{i<k}\setconstr (F_i)_i \text{ of order }\leq N+1\}$ and
\begin{equation*}
S_{+,N}((F_i)_{i<k})=
\{(U_i)_{i<k}\setconstr (\forall i)(F_i\subseteq U_i)\text{ and } (U_i)_{i<k} \text{ of order }\leq N+1\}
\end{equation*}
is $(\delta_{\Pi^0_1}^{*},\delta_{\Sigma^0_1}^{*})$-computable.
\end{lem}
\begin{proof}
Assuming $(F_i)_{i<k}\in\dom S_{+,N}$, we first deal with the case $k\geq N+2$.
Inductively in $n<k$, assume $f_i\in\Cont(X,[0,1])$ has $F_i\subseteq f_i^{-1}\{0\}$ and $K_i:=f_i^{-1}[0,2^{-1}]$ for each $i<n$.  We also assume $(F^{(n)}_i)_{i<k}$ is of order at most $N+1$ and $F_i\subseteq F^{(n)}_i$
for all $i<k$ where
$F^{(n)}_i:=(K_i,\text{ if $i<n$;}\: F_i,\text{ if $n\leq i<k$})$ ($i<k$).
Then
\begin{equation*}
S_n := \bigcup\{\bigcap_{j\leq N}F^{(n)}_{w_j}\setconstr w\in [0,k)^{N+2} \text{ injective with }w_{N+1}=n\}
\end{equation*}
is closed and disjoint from $F^{(n)}_n=F_n$.  By Urysohn's lemma there exists continuous $f_n:X\to [0,1]$
such that $F_n\subseteq f_n^{-1}\{0\}$ and $S_n\subseteq f_n^{-1}\{1\}$.  Defining $K_n$ and $(F^{(n+1)}_i)_{i<k}$ as
above, we have
$F^{(n)}_i\subseteq F^{(n+1)}_i$ for each $i<k$, and for
$w\in [0,k)^{*}$ injective with $\absval{w}\geq N+2$ we have
\begin{equation*}
\bigcap_{j<\absval{w}}F^{(n+1)}_{w_j} =
\begin{cases} \bigcap_{j<\absval{w}}F^{(n)}_{w_j}=\emptyset, &\mbox{\textrm{ if $(\forall j<\absval{w})(w_j\neq n)$}}\\
K_n\cap\bigcap_{j\neq j'<\absval{w}}F^{(n)}_{w_{j'}}, &\mbox{\textrm{ if $w_j=n$}}\end{cases} \subseteq K_n\cap S_n=\emptyset.
\end{equation*}

In step $n=k-1$ of the above induction we get $F^{(n+1)}_i=K_i$ for all $i<k$ and any injective $w\in [0,k)^{*}$
with $\absval{w}\geq N+2$ satisfies $\bigcap_{j<\absval{w}}K_{w_j}=\emptyset$.  But then for
$U_i:=f_i^{-1}[0,2^{-1})\subseteq K_i$ ($i<k$) it is clear $(U_i)_{i<k}\in S_{+,N}((F_i)_{i<k})$.
Furthermore it is clear how to obtain $\delta_{\Sigma^0_1}^{*}$-information on $(U_i)_{i<k}$.
Namely, let $F$ and $G$ be fixed computable realisers of the operations
\begin{equation*}
T^{(N)}:\subseteq\Pi^0_1(X)^{*}\times\Cont(X,\R)^{*}\to\Pi^0_1(X),((F_i)_{i<k},(f_i)_{i<n})\mapsto S_n
\end{equation*}
($\dom T^{(N)}=\{((F_i)_{i<k},(f_i)_{i<n})\setconstr n\leq k\}$) and $U:\subseteq\Pi^0_1(X)^2\rightrightarrows\Cont(X,\R)$
(from \pref{thm:urylem}), and $p=k.\langle p^{(0)},\dots,p^{(k-1)},0^{\ordomega},0^{\ordomega},\dots\rangle\in
(\delta_{\Sigma^0_1}^{*})^{-1}\{(F_i)_{i<k}\}$.  Then
\begin{equation*}
q^{(n)}:=G\langle p^{(n)},
  F\langle p,n.\langle q^{(0)},\dots,q^{(n-1)},0^{\ordomega},0^{\ordomega},\dots\rangle\rangle
 \rangle \onespace \text{($0\leq n<k$)}
\end{equation*}
are $[\delta_X\to\rho]$-names of respective $f_n$, uniformly computable from the
inputs; computability here is a matter of appropriate dovetailing.
Note for the case $k\leq N+1$ the same argument works (with $S_i=\emptyset$ for all $i<k$); in any case,
checking $(U_i)_{i<k}$ have order at most $N+1$ becomes trivial.  This completes the proof.
\end{proof}

\begin{prop}\label{prn:shrinkingconstrtwo}
For any computable metric space $X$, the operations
\begin{align*}
S_{-}: &\subseteq\Sigma^0_1(X)^{*}\crsarr\Pi^0_1(X)^{*},
(U_i)_{i<k}\mapsto \{(F_i)_{i<k}\setconstr (\forall i)(F_i\subseteq U_i)\wedge\textstyle\bigcup_i F_i=X\}\\
T: &\subseteq\Pi^0_1(X)^{*}\crsarr\Sigma^0_1(X)^{*},
(B_i)_{i<k}\mapsto \{(U_i)_{i<k}\setconstr \textstyle\bigcap_i U_i=\emptyset\wedge (\forall i)B_i\subseteq U_i\}
\end{align*}
($\dom S_{-}=\{(U_i)_{i<k}\setconstr \bigcup_i U_i=X\}$,
$\dom T=\{(B_i)_{i<k}\setconstr \bigcap_i B_i=\emptyset\}$) are
resp.~$(\delta_{\Sigma^0_1}^{*},\delta_{\Pi^0_1}^{*})$- and
$(\delta_{\Pi^0_1}^{*},\delta_{\Sigma^0_1}^{*})$-computable.
\end{prop}
\begin{proof}
Inductively in $n\leq k$, suppose
$(B^{(n)}_i)_{i<k}$, $f_i\in\Cont(X,[0,1])$ and $K_i:=f_i^{-1}[0,2^{-1}]$ ($i<n$) are such that
$B^{(n)}_i=(K_i,\textrm{ if $i<n$; }B_i,\textrm{ if $n\leq i<k$})$.  We additionally suppose that $f_i$
satisfy $B_i\subseteq f_i^{-1}\{0\}$ for all $i<n$, and that $\bigcap_{i<k}B^{(n)}_i=\emptyset$.

Then $S_n:=\bigcap_{j\in [0,k)\setminus\{n\}}B^{(n)}_j$ is closed and disjoint from $B_n$ ($=B^{(n)}_n$).  By Urysohn's lemma,
there exists continuous $f_n:X\to [0,1]$ such that $B_n\subseteq f_n^{-1}\{0\}$ and $S_n\subseteq f_n^{-1}\{1\}$.  Defining $K_n$,
$(B^{(n+1)}_i)_{i<k}$ as above in the case $n+1\leq k$, we have $B^{(n)}_i\subseteq B^{(n+1)}_i$ for all $i<k$
with
\begin{equation*}
\textstyle
\bigcap_{i<k}B^{(n+1)}_i = \bigcap_{i\leq n}K_i\cap\bigcap_{n<i<k}B_i =
\bigcap_{n\neq i<k}B^{(n)}_i\cap K_n = S_n\cap K_n = \emptyset.
\end{equation*}

By step $k$ of this induction, there exist $f_i$ such that $B_i\subseteq f_i^{-1}\{0\}$ and $K_i:=f_i^{-1}[0,2^{-1}]$
($i<k$) satisfy
$B_i\subseteq K_i$ for all $i$ and $\bigcap_i K_i=\emptyset$.  Writing $U_i:=f_i^{-1}[0,2^{-1})$ ($i<k$) we now
have $B_i\subseteq U_i\subseteq K_i$ for all $i<k$ and $\bigcap_i U_i=\emptyset$.
This establishes the computability of $T$.  Then $\dom S_{-}=\{(X\setminus B_i)_i\setconstr (B_i)_i\in\dom T\}$,
and also $(F_i)_i\in S_{-}((U_i)_i)$ iff $(X\setminus F_i)_i\in T((X\setminus U_i)_i)$.
\end{proof}

\section{Zero dimensional subsets}\label{sec:zdsubsets}
For a computable metric space $X$ and a class $\mathcal{Y}\subseteq\pwrset(X)$ of zero-dimensional or empty
subsets with $\absval{\mathcal{Y}}\leq 2^{\aleph_0}$, what information should be included (or more abstract requirements made) when specifying a representation $\delta_{\mathcal{Y}}$ of $\mathcal{Y}$?  Loosely speaking,
we would like effective versions of certain theorems concerning zero-dimensionality to hold,
without requiring `unrealistically' strong information on inputs.  While we are here far from an exposition
that would satisfactorily answer this open-ended problem, it seems a reasonable place to start is from the
definition of zero-dimensionality as presented in \pref{sec:intro}.
Specifically, as effectivisations of (\ref{zdvi}), (\ref{zdviii}), (\ref{zdii}), (\ref{zdiii}) which also depend on the subspace $Y$ in place of $X$ we (for
given $X$, $\mathcal{Y}$, $\delta_{\mathcal{Y}}$) consider computability of respective operations
$B$, $S$, $M$, $N$, defined as below.  For brevity, in case of the binary disjoint union of two sets, we often
write ``$E=C\disjunion D$" in place of ``$C\cap D=\emptyset$ and $E=C\union D$".
\begin{align*}
& B:\mathcal{Y}\crsarr(\Sigma^0_1(X)^2)^{\N}\times(\N^2)^{\N}, \onespace
S:\subseteq\Sigma^0_1(X)^{\N}\times\mathcal{Y}\crsarr\Sigma^0_1(X)^{\N},\\
& M:\subseteq X\times\Sigma^0_1(X)\times\mathcal{Y}\crsarr\Sigma^0_1(X)^2,\onespace
N:\subseteq\Pi^0_1(X)^2\times\mathcal{Y}\crsarr\Sigma^0_1(X)^2
\end{align*}
with
$\dom S=\{((U_i)_i,Y)\setconstr\bigcup_i U_i\supseteq Y\}$,
$\dom M=\{(x,U,Y)\setconstr x\in U\}$,
$\dom N=\{(A,B,Y)\setconstr A\cap B=\emptyset\}$, and
\begin{align*}
B(Y) &:=\big\{((U_i,V_i)_i,(a_i,b_i)_i)\setconstr (U_i)_i \text{ a basis for $\mathcal{T}_X$, }\:
(\forall i)Y\subseteq U_i\disjunion V_i \text{ and }\\
& \{(a_k,b_k)\setconstr k\in\N\}\text{ refined inclusion of $(U_i)_i \mathwrt \alpha$}\big\},\\
S((V_i)_i,Y) &:=\{(W_i)_i\setconstr (\forall i)W_i\subseteq V_i\wedge
\textstyle\bigcup_i W_i\supseteq Y \text{ and $(W_i)_i$ pairwise disjoint}\},\\
M(x,U,Y) &:=\{(V,W)\setconstr x\in V\subseteq U\wedge Y\subseteq V\disjunion W\},\\
N(A,B,Y) &:= \{(U,V)\setconstr A\subseteq U\wedge B\subseteq V\wedge Y\subseteq U\disjunion V\}
\end{align*}
\begin{prop}\label{prn:zdsubsets}
Let $X$ be a computable metric space and $\mathcal{Y}\subseteq\pwrset(X)$ a class of zero-dimensional or empty
subsets with representation
$\delta_{\mathcal{Y}}$.  Then \emph{(\ref{zdeffvii})}$\implies$\emph{(\ref{zdeffiii})}$\implies$\emph{(\ref{zdeffi})}
$\implies$\emph{(\ref{zdeffivprime})}.
\begin{enumerate}[label=(\roman*)]
\item\label{zdeffvii} $N:\subseteq\Pi^0_1(X)^2\times\mathcal{Y}\crsarr\Sigma^0_1(X)^2$ is computable.
\item\label{zdeffiii} $M:\subseteq X\times\Sigma^0_1(X)\times\mathcal{Y}\crsarr\Sigma^0_1(X)^2$ is computable.
\item\label{zdeffi} $B:\mathcal{Y}\crsarr(\Sigma^0_1(X)^2)^{\N}\times(\N^2)^{\N}$ is computable.
\item\label{zdeffivprime} $S:\subseteq\Sigma^0_1(X)^{\N}\times\mathcal{Y}\crsarr\Sigma^0_1(X)^{\N}$ is computable.
\end{enumerate}
\end{prop}

\begin{proof}\hfill

\noindent{(\ref{zdeffvii})$\implies$(\ref{zdeffiii}):}
If $(x,U,Y)\in\dom M$ then $(\{x\},X\setminus U,Y)\in\dom N$ and for any $(V,W)\in N(\{x\},X\setminus U,Y)$
it holds that $x\in V\subseteq X\setminus W\subseteq U$ and $Y\subseteq V\union W$
(equivalently, $(V,W)\in M(x,U,Y)$ and $X\setminus U\subseteq W$).
\newline

\noindent{(\ref{zdeffiii})$\implies$(\ref{zdeffi}):}
Consider $M^{\circ}:X\times\N\times\mathcal{Y}\crsarr\Sigma^0_1(X)^2$ defined by
\begin{equation*}
M^{\circ}(x,i,Y)=M(x,\ball{x}{2^{-i}},Y)=\{(V,W)\setconstr x\in V\subseteq\ball{x}{2^{-i}}\wedge Y\subseteq V\disjunion W\}.
\end{equation*}
If $\delta=\delta_X$ is Cauchy representation of $X$, let $G:\subseteq\Bairespc^3\to\Bairespc$ be a computable
$(\delta,\delta_{\N},\delta_{\mathcal{Y}};\delta_{\Sigma^0_1}^2)$-realiser of $M^{\circ}$,
$Z:=\dom\delta$ and
\begin{equation*}
u^{(j)}_q :=(\delta_{\Sigma^0_1}\compose\pr_1\compose G)(\cdot,j.0^{\ordomega},q):Z\to\Sigma^0_1(X)
\onespace \textrm{ ($j\in\N$, $q\in\dom\delta_{\mathcal{Y}}$).}
\end{equation*}
We now can apply \pref{lem:ctslindelof} to $u^{(j)}_q$ (with $(p^{(i)})_i$ a standard enumeration of
$\{w.w^{\ordomega}_{\absval{w}-1}\setconstr w\in A\}\subseteq Z$ for $A$ as in proof of the lemma), obtaining
$\bigcup_{i\in\N}u^{(j)}_q(p^{(i)})=X$ for every $j\in\N$, $q\in\dom\delta_{\mathcal{Y}}$.
If we denote
\begin{equation*}
 B_1:\N\times\Bairespc\to\Sigma^0_1(X)^2,(\langle i,j\rangle,q)\mapsto
(\delta_{\Sigma^0_1}^2\compose G)(p^{(i)},j.0^{\ordomega},q)
\end{equation*}
and $b^{\prime}:=\pr_1\compose B_1:\N\times\Bairespc\to\Sigma^0_1(X)$
then one can check each $b^{\prime}(\cdot,q)$ is a basis numbering.

Next, define
\begin{equation*}
\langle i,j\rangle\sqsubset'\langle n,r\rangle :\iff d(\delta(p^{(i)}),\nu(n))+2^{-j}<\nu_{\Qplus}(r).
\end{equation*}
We show
\begin{property}\label{pry:refincl}
$(\sqsubset')\subseteq\N^2$ is a c.e.~refined inclusion of $b^{\prime}(\cdot,q)\mathwrt\alpha$.
\end{property}
\noindent\textbf{Proof of \pref{pry:refincl}:}
\begin{equation*}
b^{\prime}(\langle i,j\rangle,q) = u^{(j)}_q(p^{(i)})
= (\pr_1\compose\delta_{\Sigma^0_1}^2\compose G)(p^{(i)},j.0^{\ordomega},q) \in
(\pr_1\compose M^{\circ})(\delta(p^{(i)}),j,\delta_{\mathcal{Y}}(q))
\end{equation*}
implies $\delta(p^{(i)})\in b^{\prime}(\langle i,j\rangle,q) \subseteq \ball{\delta(p^{(i)})}{2^{-j}}$, where
the latter set is included in $\alpha\langle n,r\rangle$ if $\langle i,j\rangle\sqsubset'\langle n,r\rangle$.

Secondly, for $p\in\Bairespc$ and $N\in\N$ let $p^N$ denote the prefix $p_0\dots p_{N-1}$ of $p$.  We let
$s\in\N$, $y:=\nu(\pr_1 s)$, $r:=\nu_{\Qplus}(\pr_2 s)$, $x\in\alpha(s)$, $p\in\delta^{-1}\{x\}$ and define
\begin{equation*}
H^{(j)}_q:\subseteq\Bairespc\to\Bairespc,p\mapsto (\pr_1\compose G)(p,j.0^{\ordomega},q).
\end{equation*}
$H^{(j)}_q$ is a continuous
$(\id_{\Bairespc}|^Z,\delta_{\Sigma^0_1})$-realiser of $u^{(j)}_q$ ($q\in\dom\delta_{\mathcal{Y}}$).  Fix $j$
with $d(x,y)+2^{-j}<r$, $l\in\N$ with
$H^{(j)}_q(p)_l\geq 1$ and $x\in\alpha\left(H^{(j)}_q(p)_l-1\right)$ and
$N\in\N$ with $H^{(j)}_q(\dom\delta\cap p^N.\Bairespc)\subseteq H^{(j)}_q(p)^{l+1}.\Bairespc$.  Any
$p'\in\dom\delta\cap p^N.\Bairespc$ satisfies
\begin{equation*}
u^{(j)}_q(p') = (\delta_{\Sigma^0_1}\compose H^{(j)}_q)(p') \supseteq
\bigcup\{\alpha\left(H^{(j)}_q(p)_{l'}-1 \right)\setconstr l'\leq l\wedge H^{(j)}_q(p)_{l'}\geq 1\}
\end{equation*}
where the last set contains the point $x$.
By density of $(p^{(i)})_i\subseteq\dom\delta$,
pick $i$ with $p^{(i)}\in\dom\delta\cap p^N.\Bairespc$ and $d(\delta(p^{(i)}),y)+2^{-j}<r$.  Then
\begin{equation*}
x\in u^{(j)}_q(p^{(i)}) = b^{\prime}(\langle i,j\rangle,q) \:\text{ ($\subseteq\ball{\delta(p^{(i)})}{2^{-j}}$)}
\end{equation*}
and $\langle i,j\rangle\sqsubset' s$.  This completes the proof of \pref{pry:refincl}.

Finally we show $B:\mathcal{Y}\crsarr(\Sigma^0_1(X)^2)^{\N}\times(\N^2)^{\N}$ is computable.  Fix $h\in\trfns$
such that $\img h=\{\langle a,b\rangle\setconstr a\sqsubset' b\}$ and consider as a realiser the map
$I:\subseteq\Bairespc\to\Bairespc$ defined by
\begin{equation*}
I(q):=\langle \langle r^{(0)},r^{(1)},\dots\rangle,
             \langle s^{(0)},s^{(1)},\dots\rangle\rangle
\end{equation*}
where $r^{(\langle i,j\rangle)}=G(p^{(i)},j.0^{\ordomega},q)$ and
$s^{(k)}=\langle\pr_1 h(k).0^{\ordomega},\pr_2 h(k).0^{\ordomega}\rangle$ ($i,j,k\in\N$).  That is, take
$(U_i,V_i):=(\delta_{\Sigma^0_1}^2\compose G)(p^{(\pr_1 i)},\pr_2 i.0^{\ordomega},q)=B_1(i,q)$ and
$\langle a_i,b_i\rangle:=h(i)$ for each $i$, so
$(U_i)_i$ gives the basis numbering $b^{\prime}(\cdot,q)$ and $(a_i,b_i)_i$ gives the relation
$\sqsubset'$ independent of $q$.

For a fixed $q\in\dom\delta_{\mathcal{Y}}$, observe
$(U_i,V_i)\in M^{\circ}(\delta(p^{(\pr_1 i)}),\pr_2 i,\delta_{\mathcal{Y}}(q))$ implies
$\delta(p^{(\pr_1 i)})\in U_i\subseteq\ball{\delta(p^{(\pr_1 i)})}{2^{-\pr_2 i}}$ and
$\delta_{\mathcal{Y}}(q)\subseteq U_i\disjunion V_i$.
Then $((U_i,V_i)_i,(a_i,b_i)_i)\in (B\compose\delta_{\mathcal{Y}})(q)$ trivially.
\newline

\noindent{(\ref{zdeffi})$\implies$(\ref{zdeffivprime}):}
This proof derives from \cite[\S 26.II, Thm 1]{Kuratowski}.
Assume we are given $((V_i)_i,Y)\in\dom S$, $((T_i,U_i)_{i\in\N},(a_k,b_k)_k)\in B(Y)$
and $\langle p^{(0)},p^{(1)},\dots\rangle \in (\delta_{\Sigma^0_1(X)}^{\ordomega})^{-1}\{(V_i)_{i\in\N}\}$.
For each $i$ enumerate
$(0 \text{ for each } j \mathst p^{(i)}_j=0; a_k+1 \text{ for any }j,k \mathst p^{(i)}_j=b_k+1)$.  By definition of $B$, $\sqsubset'$
defined in $(a_k,b_k)_k$ is a refined inclusion of $(T_i)_i$ with respect to $\alpha$,
so for any $b\in\N$ and $x\in X$ there exists $k\in\N$ such that $x\in\alpha(b)$ implies $x\in T_{a_k}$ and $b_k=b$.
Since
$\emptyset\neq\img\alpha\not\ni\emptyset$, this implies $\{b_k\setconstr k\in\N\}=\dom\alpha=\N$, so output is
infinite for each $i$ --- say this output is $q^{(i)}\in\Bairespc$.

Now, let $(T_{i,j},U_{i,j}):=\begin{cases}
(\emptyset,\emptyset), &\mbox{\textrm{ if $q^{(i)}_j=0$}}\\
(T_a,U_a), &\mbox{\textrm{ if $q^{(i)}_j=a+1$}}\end{cases}$ ($j\in\N$).  We have
$(\forall i)V_i=\bigcup_j T_{i,j}$.
Also
\begin{equation*}
W_{i,j}^{*}:=T_{i,j}\cap\bigcap_{\langle k,l\rangle<\langle i,j\rangle}U_{k,l} \subseteq
T_{i,j}\cap\bigcap_{\langle k,l\rangle<\langle i,j\rangle}(X\setminus T_{k,l})
= T_{i,j}\setminus\bigcup_{\langle k,l\rangle<\langle i,j\rangle}T_{k,l} \onespace\text{($i,j\in\N$)}
\end{equation*}
are pairwise disjoint with $\delta_{\Sigma^0_1}$-information available uniformly in $i$, $j$ and the inputs.
Then $W_i:=\bigcup_j W_{i,j}^{*}$ ($\subseteq V_i$, $i\in\N$) are pairwise disjoint with a
$\delta_{\Sigma^0_1}^{\ordomega}$-name of $(W_i)_i$ available.  Finally, any
$x\in (\bigcup_i V_i)\setminus(\bigcup_{i'}W_{i'})=(\bigcup_{i,j}T_{i,j})\setminus(\bigcup_{i,j}W_{i,j}^{*})$
has $x\not\in Y$ by an argument we now elaborate.  First, denote
$Z_k:=T_{\pr_1 k,\pr_2 k}$ and $Z_k^{*}:=W_{\pr_1 k,\pr_2 k}^{*}$ ($k\in\N$).
Then one can check
\begin{equation}\label{eq:relredind}
Y\cap\textstyle\bigcup_{k<l}Z_k\subseteq\bigcup_{k<l}Z_k^{*}
\end{equation}
inductively.  Namely, assume (\ref{eq:relredind}) for some
$l\in\N$ (this is trivially true for $l=0$).  Then
$Z_k^{*} = Z_k\cap\bigcap_{k'<k}U_{\pr_1 k',\pr_2 k'}\subseteq Z_k\setminus\bigcup_{k'<k}\cl{Z_{k'}}$ and
$(\forall k')Y\subseteq T_{k'}\union U_{k'}$ imply
\begin{equation*}
Y\cap Z_k\subseteq (Y\cap Z_k^{*})\union\textstyle\bigcup_{k'<k}(Y\setminus U_{\pr_1 k',\pr_2 k'}) \subseteq
Z_k^{*}\union\bigcup_{k'<k}(Y\cap Z_{k'})
\end{equation*}
for all $k\in\N$, in particular
\begin{equation*}
Y\cap\textstyle\bigcup_{k\leq l}Z_k\subseteq
\bigcup_{k\leq l}\left(Z_k^{*}\union\bigcup_{k'<k}(Y\cap Z_{k'})\right)
= \bigcup_{k\leq l}Z_k^{*}\union\bigcup_{k<l}(Y\cap Z_k)\subseteq \bigcup_{k\leq l}Z_k^{*}
\end{equation*}
by inductive assumption.  So, we established
$Y=Y\cap\bigcup_i V_i=Y\cap\bigcup_{i,j}T_{i,j}=Y\cap\bigcup_k Z_k \subseteq
\bigcup_k Z_k^{*}=\bigcup_{i,j}W_{i,j}^{*}=\bigcup_i W_i\subseteq\bigcup_i V_i$, and in particular $(W_i)_i$
is a cover of $Y$.  This proves computability of $S$.
\end{proof}

At least two implications in Proposition \ref{prn:zdsubsets} could be improved to results concerning
Weihrauch reducibility (\cite{BrattkaGherardiEffChoiceBdd}) between the mentioned operations.  If e.g.~each
operation
$M(\cdot,\cdot,Y)$ ($Y\in\mathcal{Y}$) is guaranteed to possess realisers of a given represented class, then a
corresponding enriched representation $\delta_{\mathcal{Y},M}$ can also be defined.  For the purposes of the
present paper, we do not study these notions further; in particular, we have not separated the conditions of computability for $N$, $M$, $B$, $S$.  We mainly consider a situation where
\pref{prn:zdsubsets} is applied to $Y=X$ (in \pref{sec:zdspcs} and thereafter in Sections \ref{sec:repnsec} and
\ref{sec:retractsec}).

\section{Zero-dimensional spaces}\label{sec:zdspcs}

Less broadly than in \pref{sec:zdsubsets}, one can ask what constitutes a useful nonuniform definition of effectively zero-dimensional computable metric space; more generally, this might be addressed for closed effectively separable subspaces.
In this paper we consider the problem for $Y=X$ only\footnote{The subspace case could subsequently be treated
following \pref{sec:dimsec} to an effectivisation of the theorem on closed subspaces
\cite[Thm 7.1.8]{Engelking}, but we will not do that here.}.
We consider computability of the following operations, again based on (\ref{zdii})-(\ref{zdviii}) in
\pref{sec:intro}:
\begin{align*}
& \tilde{S}=\tilde{S}^X :\Sigma^0_1(X)^{\N}\crsarr \Sigma^0_1(X)^{\N},\onespace
R=R^X :\subseteq\Sigma^0_1(X)^{\N}\to \Delta^0_1(X)^{\N},\\
& M :\subseteq X\times\Sigma^0_1(X)\crsarr \Delta^0_1(X),\onespace
N :\subseteq\Pi^0_1(X)^2\crsarr\Delta^0_1(X)
\end{align*}
with $\dom R^X=\{(V_i)_i\setconstr (V_i)_i\text{ pairwise disjoint with } \textstyle\bigcup_i V_i=X\}$,
$\dom M=\{(x,U)\setconstr x\in U\}$, $\dom N=\{(A,B)\setconstr A\cap B=\emptyset\}$,
$R^X((V_i)_i)=(V_i)_i$ and
\begin{align*}
\tilde{S}^X((U_i)_i) & =\{(W_i)_i\setconstr (W_i)_i
\text{ pairwise disjoint with $W_i\subseteq U_i$ and } \textstyle\bigcup_i W_i=\bigcup_i U_i\},\\
 M(x,U) &=\{W\in\Delta^0_1(X)\setconstr x\in W\subseteq U\},\\
N(A,B) &=\{W\setconstr A\subseteq W\wedge B\subseteq X\setminus W\}.
\end{align*}
Except for $R^X$ these operations are related to those defined in \pref{sec:zdsubsets}.  For instance,
label temporarily the new operation as $N'$ and suppose $X$ is zero-dimensional, with $\mathcal{Y}\ni X$ and
some computable $p\in\dom\delta_{\mathcal{Y}}$ such that $\delta_{\mathcal{Y}}(p)=X$.  Then
\begin{equation*}
N\text{ computable}\implies N(\cdot,\cdot,X)\text{ computable}\iff N'\text{ computable}.
\end{equation*}
If also $\mathcal{Y}=\{X\}$, we can derive full equivalence (using definition of product representations).
The situation is similar for the operations $M$, $B$ and $S$ (here compared to a
suitable restriction of $\tilde{S}^X$), e.g.~for $B$ this leads to the condition (\ref{effzdii}) in the
following
\begin{prop}\label{prn:zerodspc}
Let $X$ be a computable metric space.  Then the following conditions are equivalent:
\begin{enumerate}
\item\label{effzdii} There exist computable $b:\N\to\Delta^0_1(X)$ and c.e.~refined inclusion of
$b$ with respect to $\alpha$
such that $\mathcal{B}:=\img b$ is a basis for $\mathcal{T}_X$.
\item\label{effzdvii} The operation $N$ is computable.
\item\label{effzdviii} The operation $\cl{M}:\subseteq X\times\Pi^0_1(X)\crsarr\Delta^0_1(X),(x,A)\mapsto
N(\{x\},A)$ is computable, where $\dom\cl{M}=\{(x,A)\setconstr x\not\in A\}$.
\item\label{effzdiii} The operation $M$ is computable.

\item\label{effzdiv} The operation $\dot{C}^{\ordomega}:\subseteq\Sigma^0_1(X)^{\N}\crsarr
\Sigma^0_1(X)^{\N}\times\Bairespc$ is computable, where
\begin{equation*}
\dot{C}^{\ordomega}((U_i)_i)=
\{((W_i)_i,r)\setconstr (W_i)_i \text{ pairwise disjoint, }
\textstyle\bigcup_i W_i=X,\, (\forall i)W_i\subseteq U_{r_i}\}
\end{equation*}
and $\dom \dot{C}^{\ordomega}=\{(U_i)_i\setconstr \bigcup_i U_i=X\}$.
\item\label{effzdv} The operation $C^{\ordomega}:=\tilde{S}|_{\dom\dot{C}^{\ordomega}}$ is computable.
\item\label{effzdix} The operation $R^X\compose C^{\ordomega}\compose L'$ is computable for every computable metric space $(Z,d',\nu')$.

\item\label{effzdivprime} The operation $C^{*}:\subseteq\Sigma^0_1(X)^{*}\crsarr\Sigma^0_1(X)^{*}$ is
computable where
\begin{equation*}
C^{*}((U_i)_{i<k})=\{(V_j)_{j<k}\setconstr
(\forall i<k)(V_i\subseteq U_i), \,\textstyle\bigcup_j V_j=X,\, (\forall i,j<k)\left(
i\neq j\smash{\implies}V_i\cap V_j=\emptyset \right)\}
\end{equation*}
and $\dom C^{*}=\{(U_i)_{i<n}\setconstr n\in\N\wedge\bigcup_i U_i=X\}$.
\end{enumerate}
\end{prop}

\noindent Note the conditions (\ref{effzdvii}), (\ref{effzdviii}) and (\ref{effzdivprime}) correspond to definitions of
large and small inductive dimension, and (loosely speaking) of covering dimension, respectively.
\begin{proof}\hfill

\noindent{(\ref{effzdii})$\implies$(\ref{effzdv}):}
Follows from \pref{prn:zdsubsets}(\ref{zdeffi})$\implies$(\ref{zdeffivprime}).
\begin{rem}
A simpler effectivization of \cite[\S 26.II, Thm 1]{Kuratowski} shows that $\tilde{S}$ is
$(\delta_{\Sigma^0_1}^{\ordomega},\delta_{\Sigma^0_1}^{\ordomega})$-computable under the same assumption.
\end{rem}
\noindent{(\ref{effzdv})$\implies$(\ref{effzdivprime}):}
Trivial.  See \pref{lem:effdimlem}((\ref{effdimi})$\implies$(\ref{effdimii})) for an extension.

\noindent{(\ref{effzdivprime})$\implies$(\ref{effzdvii}):}
Consider arbitrary disjoint closed $A,B\subseteq X$.  Then $(U_i)_{i<2}=(X\setminus A,X\setminus B)$ has
$\bigcup_i U_i=X$ and any $(W_i)_{i<2}\in C^{*}((U_i)_i)$
satisfies $W_1=X\setminus W_0\supseteq A$ and $X\setminus W_1\supseteq B$, hence $W_1\in N(A,B)$.
Also $N$ is computable using
$\delta_{\Delta^0_1}$-information on $W_1$ (more formally, use the second projection from
$R(W_0,W_1,\emptyset,\emptyset,\dots)$).

\noindent{(\ref{effzdvii})$\implies$(\ref{effzdiii}):}
Follows from \pref{prn:zdsubsets}(\ref{zdeffvii})$\implies$(\ref{zdeffiii}).

\noindent{(\ref{effzdviii})$\iff$(\ref{effzdiii}):}
$(x,U)\in\dom M \iff (x,X\setminus U)\in\dom\cl{M}$ with
$\cl{M}(x,X\setminus U) = M(x,U)$ for any such $x$, $U$.

\noindent{(\ref{effzdiii})$\implies$(\ref{effzdii}):}
Follows from \pref{prn:zdsubsets}(\ref{zdeffiii})$\implies$(\ref{zdeffi}).

\noindent{(\ref{effzdv})$\implies$(\ref{effzdix}):}
Use \pref{lem:ctslindelof}, the closure scheme of composition (for partial functions) and computability of
$R^X:\subseteq\Sigma^0_1(X)^{\N}\to\Delta^0_1(X)^{\N}$.  Namely, the latter has a computable realiser
$F:\subseteq\Bairespc\to\Bairespc$ defined by
\begin{equation*}
F(p)\langle i,2\langle k,j\rangle+z\rangle =
\begin{cases}
p_{\langle i,\langle k,j\rangle\rangle}, &\mbox{\textrm{ if $z=0$}}\\
p_{\langle k,j\rangle}, &\mbox{\textrm{ if $z=1\wedge k\neq i$}}\\
p_{\langle i+1,j\rangle},& \mbox{\textrm{ if $z=1\wedge k=i$}}\end{cases}.
\end{equation*}
Then $F\langle p^{(0)},p^{(1)},\dots\rangle=\langle \langle p^{(0)},q^{(0)}\rangle,\langle p^{(1)},q^{(1)}\rangle,\dots\rangle$ where
$\{q^{(i)}_j-1\setconstr j\in\N,\: q^{(i)}_j\geq 1\} =
\{p^{(k)}_j-1\setconstr j,k\in\N,\: k\neq i,\: p^{(k)}_j\geq 1\}$ for each $i$.

\noindent{(\ref{effzdix})$\implies$(\ref{effzdv}):}
Take $Z=\N$, $\nu'=\id_{\N}$; then $L'$ from \pref{lem:ctslindelof} is the identity on $\Sigma^0_1(X)^{\N}$,
and $R^X$ has a computable left-inverse.

\noindent{(\ref{effzdv})$\iff$(\ref{effzdiv}):}
Essentially trivial.  See \pref{lem:effdimlem}((\ref{effdimi})$\iff$(\ref{effdimiv})) for an extension.
\end{proof}

\section{Covering dimension}\label{sec:dimsec}

For a normal topological space $X$ and $n\in\{-1\}\disjunion\N$, write $\dim X\leq n$ if any finite open cover
of $X$ has a finite open refinement of order at most $n+1$; write $\dim X=n$ if $\dim X\leq k$ fails exactly
when $k<n$, or $\dim X=\infty$ if $\dim X\leq k$ fails for all $k\geq -1$.  $\dim X$ is the
(Lebesgue-{\v C}ech) \emph{covering dimension}.  We first recall several classically equivalent forms of the
definition.
\begin{thm}(\cite[Thm 4.3.5]{vanMill})\label{thm:covdimeqv}
For a nonempty separable metric space $X$ and $n\in\N$, the following conditions are equivalent:
\begin{enumerate}
\item\label{covdimi} $\dim X\leq n$ ($\iff \dim X<n+1$),
\item\label{covdimii}
every open cover $\mathcal{U}$ of $X$ has a locally finite closed refinement $\mathcal{V}$ with order
$\leq n+1$,
\item\label{covdimiii}
every open cover $\mathcal{U}$ of $X$ has an open refinement $\mathcal{V}$ with order $\leq n+1$,
\item\label{covdimiv}
every open cover $\mathcal{U}$ of $X$ has a closed shrinking $\mathcal{V}$ with order $\leq n+1$,
\item\label{covdimv}
every open cover $\mathcal{U}$ of $X$ has an open shrinking $\mathcal{V}$ with order $\leq n+1$,
\item\label{covdimvi}
every finite open cover $\mathcal{U}$ of $X$ has a closed shrinking $\mathcal{V}$ with order $\leq n+1$,
\item\label{covdimvii}
every finite open cover $\mathcal{U}$ of $X$ has an open shrinking $\mathcal{V}$ with order $\leq n+1$.
\end{enumerate}
\end{thm}
Leaving $N\in\N$ fixed we next consider some effective versions of several such conditions, including
(\ref{covdimi}), (\ref{covdimiii}), (\ref{covdimv}), (\ref{covdimvi}) and (\ref{covdimvii}) above.  Define
$C^{\sigma}:\subseteq\Sigma^0_1(X)^{\sigma}\crsarr\Sigma^0_1(X)^{\sigma}$,
$\dot{C}^{\sigma}: \subseteq\Sigma^0_1(X)^{\sigma}\crsarr\Sigma^0_1(X)^{\sigma}\times\N^{\sigma}$
($\sigma={*},\ordomega$),
$\cl{C}:\subseteq\Sigma^0_1(X)^{*}\crsarr\Pi^0_1(X)^{*}$ and
$\dot{\cl{C}}: \subseteq\Sigma^0_1(X)^{*}\crsarr\Pi^0_1(X)^{*}\times\N^{*}$ by
\begin{align*}
&C^{\sigma}((U_i)_i) = \{(W_i)_i\setconstr (W_i)_i \textrm{ shrinking of $(U_i)_i$ of order $\leq N+1$}\}\\
&\dot{C}^{*}((U_i)_{i<k}) =
\{((W_j)_{j<l},r)\setconstr \absval{r}=l, (\forall j<l)W_j\subseteq U_{r_j}, (W_j)_j \textrm{ cover of order $\leq N+1$}\}\\
&\dot{C}^{\ordomega}((U_i)_{i\in\N}) = \{((W_i)_i,r)\setconstr (\forall j)W_j\subseteq U_{r_j}, (W_j)_j \textrm{ cover of order $\leq N+1$}\}\\
&\cl{C}((U_i)_{i<k})=\{(F_i)_{i<k}\setconstr (F_i)_i \textrm{ shrinking of $(U_i)_i$ of order $\leq N+1$}\}\\
&\dot{\cl{C}}((U_i)_{i<k})=\{((F_j)_{j<l},r)\setconstr \absval{r}=l, (\forall j<l)F_j\subseteq U_{r_j},
(F_j)_j \textrm{ cover of order $\leq N+1$}\}
\end{align*}
Here $\dom C^{\sigma}=\dom\dot{C}^{\sigma}=\{(U_i)_i\in\Sigma^0_1(X)^{\sigma}\setconstr\bigcup_i U_i=X\}$ and
$\dom\cl{C}=\dom\dot{\cl{C}}=\dom C^{*}$.
The following lemma includes an effective version of \cite[Thm 7.1.7]{Engelking} and extends parts of
\pref{prn:zerodspc}.
\begin{lem}\label{lem:effdimlem}
For a computable metric space $X$ and $N\in\N$, consider the following conditions.
\begin{enumerate}
\item\label{effdimi} $C^{\ordomega}$ is
$(\delta_{\Sigma^0_1}^{\ordomega},\delta_{\Sigma^0_1}^{\ordomega})$-computable.

\item\label{effdimiv} $\dot{C}^{\ordomega}$ is
$(\delta_{\Sigma^0_1}^{\ordomega},[\delta_{\Sigma^0_1}^{\ordomega},\id_{\Bairespc}])$-computable.

\item\label{effdimii} $C^{*}$ is $(\delta_{\Sigma^0_1}^{*},\delta_{\Sigma^0_1}^{*})$-computable.
\item\label{effdimiii}
$\dot{C}^{*}$ is $(\delta_{\Sigma^0_1}^{*},[\delta_{\Sigma^0_1}^{*},\delta_{\N^{*}}])$-computable.
\item\label{effdimv} $\cl{C}$ is $(\delta_{\Sigma^0_1}^{*},\delta_{\Pi^0_1}^{*})$-computable.
\item\label{effdimvi} $\dot{\cl{C}}$ is
$(\delta_{\Sigma^0_1}^{*},[\delta_{\Pi^0_1}^{*},\delta_{\N^{*}}])$-computable.
\end{enumerate}
Then \emph{(\ref{effdimii})}, \emph{(\ref{effdimiii})}, \emph{(\ref{effdimv})} and \emph{(\ref{effdimvi})} are equivalent.
Also  \emph{(\ref{effdimi})}$\iff$\emph{(\ref{effdimiv})$\implies$(\ref{effdimii})}.
\end{lem}
\begin{proof}\hfill

\noindent{(\ref{effdimi})$\implies$(\ref{effdimiv}):}
trivial (take $r=\id_{\N}$);
\noindent{(\ref{effdimii})$\implies$(\ref{effdimiii}):}
take $l=k$, $r=0 1\dots (k-1)\in\N^{*}$;\\
\noindent{(\ref{effdimiii})$\implies$(\ref{effdimii}):}
let $((V_j)_{j<l},r)\in\dot{C}^{*}((U_i)_{i<k})$ and $W_i:=\bigcup\{V_j\setconstr j<l,\: r_j=i\}$ ($i<k$).
Then
\begin{equation*}
\bigcap_{j\leq N+1}W_{i_j} =
\bigcup\{\bigcap_{m=0}^{N+1}V_{j_m}\setconstr \vec{j}\in [0,l)^{N+2}\wedge (\forall m\leq N+1)r_{j_m}=i_m\} =
\emptyset
\end{equation*}
for any distinct indices $i_0,\dots,i_{N+1}<k$ (for, any such $\vec{j}$ is injective and
$(V_i)_{i<l}$ has order at most $N+1$).\\
\noindent{(\ref{effdimiv})$\implies$(\ref{effdimi}):}
Let $((V_j)_{j\in\N},r)\in\dot{C}^{\ordomega}((U_i)_{i\in\N})$ and
$W_i:=\bigcup\{V_j\setconstr j\in\N,\: r_j=i\}$ ($i\in\N$).\\
\noindent{(\ref{effdimi})$\implies$(\ref{effdimii}):}
$(W_i)_{i\in\N}\in C^{\ordomega}(U_0,\dots,U_{k-1},\emptyset,\emptyset,\dots)$ implies $W_i=\emptyset$ for all $i\geq k$.\\
\noindent{(\ref{effdimv})$\iff$ (\ref{effdimvi}):}
same as (\ref{effdimii})$\iff$ (\ref{effdimiii}).
\noindent{(\ref{effdimii})$\implies$(\ref{effdimv}):}
if $(V_i)_{i<k}\in C^{*}((U_i)_{i<k})$, applying \pref{prn:shrinkingconstrtwo} gives in particular $(F_i)_{i<k}\in S_{-}((V_i)_{i<k})$
which is a closed cover with $(\forall i)F_i\subseteq V_i\subseteq U_i$.  Any string of indices $w\in[0,k)^{*}$ has
$\bigcap_{j<\absval{w}}F_{w_j} \subseteq\bigcap_{j<\absval{w}}V_{w_j}$, so $(F_i)_{i<k}$ is of order at most
$N+1$ also.\\
\noindent{(\ref{effdimv})$\implies$(\ref{effdimii}):}
Given a finite open cover $(V_i)_{i<k}$ and $(F_i)_{i<k}\in \cl{C}((V_i)_{i<k})$,
apply \pref{lem:swellbddord} to obtain
$(U_i)_{i<k}\in S_{+,N}((F_i)_{i<k})$.  By definition, $(F_i)_{i<k}$, $(U_i)_{i<k}$ both have order at most
$N+1$, and $(U_i)_{i<k}$ is a cover since $(F_i)_{i<k}$ is.  By computability of $\cl{C}$ and $S_{+,N}$ we
obtain $\delta_{\Sigma^0_1}^{*}$-information on $(U_i)_{i<k}$.
\end{proof}

In view of the results of \pref{lem:effdimlem} (and the classical definition of covering dimension) it seems
reasonable to make the following
\begin{defi}
Let $(X,d,\nu)$ be a computable metric space.  If Condition (\ref{effdimiii}) of \pref{lem:effdimlem}) holds (equivalently, (\ref{effdimii})),
say $X$ is \emph{effectively of covering dimension at most $N$}.
\end{defi}

Further equivalent conditions for $\dim X\leq n$ can also be investigated.  Here we will restrict ourselves
to considering a couple of operations of fixed arity $N+2$.  If $X$ is a computable metric space and $N\in\N$,
define $C:\subseteq\Sigma^0_1(X)^{N+2}\crsarr\Sigma^0_1(X)^{N+2}$ by
\begin{equation*}
C((U_i)_{i\leq N+1}) =\{(W_i)_{i\leq N+1}\setconstr (\forall i)(W_i\subseteq U_i) \wedge
\textstyle\bigcup_i W_i=X\wedge \bigcap_i W_i=\emptyset\};
\end{equation*}
here $\dom C=\{(U_i)_{i\leq N+1}\setconstr \bigcup_i U_i=X\}$.
Then we have the following (cf.~the classical results \cite[Lemma 7.2.13, Cor 7.2.14]{Engelking})
\begin{thm}\label{thm:covdimshort}
Let $X$ be a computable metric space and $N\in\N$.  Then the following are equivalent:
\begin{enumerate}
\item\label{covdimshorti} $C^{*}$ is computable.
\item\label{covdimshortii} $C$ is $(\delta_{\Sigma^0_1(X)}^{N+2},\delta_{\Sigma^0_1(X)}^{N+2})$-computable.
\item\label{covdimshortiii} $D$ is computable, where $D:\subseteq\Pi^0_1(X)^{N+2}\crsarr\Pi^0_1(X)^{N+2}$ with
$\dom D:=\{(B_i)_{i\leq N+1}\setconstr \bigcap_i B_i=\emptyset\}$ and
\begin{equation*}
D((B_i)_{i\leq N+1})\ni (F_i)_{i\leq N+1} :\iff (\forall i)(B_i\subseteq F_i)\wedge\textstyle\bigcup_i F_i=X
\wedge \bigcap_i F_i=\emptyset.
\end{equation*}
\end{enumerate}
\end{thm}
\begin{proof}\hfill

\noindent{(\ref{covdimshorti})$\implies($\ref{covdimshortii}):}
Any realiser of $C^{*}$, given a $\delta_{\Sigma^0_1(X)}^{*}$-name of $(U_i)_{i\leq N+1}$, computes a name of some shrinking $(W_i)_{i\leq N+1}$ with order at most $N+1$, i.e.~$\bigcap_{i\leq N+1}W_i=\emptyset$.  Since
$\delta_{\Sigma^0_1(X)}^{*}|^{\Sigma^0_1(X)^{N+2}}\equiv \delta_{\Sigma^0_1(X)}^{N+2}$ the result follows.

\noindent{(\ref{covdimshortii})$\implies$(\ref{covdimshorti}):}
Given $(U_i)_{i<m}\in\dom C^{*}$, note it is trivially a shrinking of itself of order at most $N+1$ if $m<N+2$
(then no $\vec{i}\in [0,m)^{N+2}$ is injective).  If $m=N+2$, clearly it is enough to apply $C$.  If $m>N+2$ we
can apply $C$ several times, as follows.  First, given $(U_i)_{i<m}\in\dom C^{*}$, compute some
$(A_l)_{l<L}\subseteq E(\N)$ enumerating all $A\subseteq [0,m)$ with $\absval{A}=N+1$; this can be done
computably in $m$, $N$.
Define
$H:\subseteq\Sigma^0_1(X)^{*}\times E(\N)\crsarr\Sigma^0_1(X)^{*}$ by
$\dom H=\{((U_i)_{i<m},A)\setconstr \bigcup_i U_i=X,\: A\subseteq [0,m),\: \absval{A}=N+1\}$ and
\begin{equation*}
\begin{split}
H(U_0\dots U_{m-1},A) =
\{(V_i)_{i<m}\setconstr (\exists W\in\Sigma^0_1(X)) (V_{i_0}\dots V_{i_N};W)\in
 C(U_{i_0}\dots U_{i_N};\textstyle\bigcup_{A\not\ni i<m}U_i),\\
(\forall j<N)(i_j<i_{j+1}),\: A=\{i_j\setconstr j\leq N\}
\:\text{and}\:(\forall i<m)(i\not\in A\implies V_i=W\cap U_i)\}.
\end{split}
\end{equation*}
One checks $H$ is computable, since $C$, binary union and intersection for open sets and relevant operations
with finite sets are computable.  In particular, the (inner to outer) composition of $H(\cdot,A_k)$ ($k<L$) is
computable.

We write $V^{(0)}_i := U_i$ ($i<m$) and $(V^{(k+1)}_i)_{i<m} \in H((V^{(k)}_i)_{i<m},A_k)$ for $k<L$.
Then it is sufficient to prove the following property holds inductively:
\begin{property}\label{pry:shrord}
$(V^{(k)}_i)_{i<m}$ is a shrinking of $(V^{(0)}_i)_{i<m}$ with
$V^{(k)}_i\cap\bigcap_{j\in A_l}V^{(k)}_j=\emptyset$ if $l<k$, $m>i\not\in A_l$.
\end{property}
Trivially \pref{pry:shrord} holds for $k=0$.  For the inductive case, any $i<m$ has either $i\not\in A_k$
(so $V^{(k+1)}_i=W\cap V^{(k)}_i\subseteq V^{(k)}_i$, where $W$ depends on $k$) or $i\in A_k$, say $i=i_j$
(where $i_0<\dots<i_N$ are all the elements of $A_k$).  In the latter case, $V^{(k+1)}_i=V^{(k+1)}_{i_j}
\subseteq V^{(k)}_{i_j}=V^{(k)}_i$.  Also,
\begin{equation*}
\textstyle\bigcup_{i<m}V^{(k+1)}_i =
\bigcup_{i\in A_k}V^{(k+1)}_i\union (\bigcup_{m>i\not\in A_k}W\cap V^{(k)}_i) =
\bigcup_{i\in A_k}V^{(k+1)}_i\union W=X,
\end{equation*}
so $(V^{(k+1)}_i)_{i<m}$ is a shrinking of $(V^{(k)}_i)_{i<m}$.
Now consider $A_l$
where $l<k$; for $m>i\not\in A_l$ we have
\begin{equation*}
\textstyle\bigcap_{j\in A_l}V^{(k+1)}_j\cap V^{(k+1)}_i
\subseteq \bigcap_{j\in A_l}V^{(k)}_j\cap V^{(k)}_i=\emptyset.
\end{equation*}
If instead $l=k$ and $i\not\in A_k$
then
\begin{equation*}
\textstyle\bigcap_{j\in A_k}V^{(k+1)}_j\cap V^{(k+1)}_i\subseteq \bigcap_{j\leq N}V^{(k+1)}_{i_j}\cap W=\emptyset.
\end{equation*}

Using the above induction, after $L$ steps we have dealt with each $A_k$ ($k<L$).
But then \pref{pry:shrord} means $(V^{(L)}_i)_{i<m}$ is a shrinking of $(U_i)_{i<m}$ of order at most $N+1$.

\noindent{(\ref{covdimshortii})$\iff$(\ref{covdimshortiii}):}
$\dom D=\{(B_i)_{i\leq N+1}\setconstr (X\setminus B_i)_{i\leq N+1}\in\dom C\}$, with
$(F_i)_{i\leq N+1}\in D((B_i)_{i\leq N+1})$ iff $(X\setminus F_i)_{i\leq N+1}\in C((X\setminus B_i)_{i\leq N+1})$.
\end{proof}

\section{Compact subsets and an application}\label{sec:repnsec}
In this section our intention is to present some consequences of assuming that $X$ effectively has
covering dimension at most $0$.  In fact, as we will be dealing with total boundedness it is convenient to make
a stronger assumption than in Sections \ref{sec:zdspcs} and \ref{sec:dimsec}, incorporating effective
compactness.  For working with computability of compact subsets we will assume familiarity with
\cite{BrattkaPresser}, though our notation will be slightly different.  Any $w\in\N^{*}$ codes an
\emph{ideal cover}, namely the finite collection of open sets $\alpha(w_i)$ ($i<\absval{w}$).  Informally, a
$\deltac$-name of $K\in\mathcal{K}(X)$ is an unpadded list consisting of $\langle w\rangle$ for every ideal
cover $w$ which covers $K$.
\begin{defi}
Ideal covers $u,v\in\N^{*}$ are \emph{formally disjoint} if
\begin{equation*}
(\forall i<\absval{u})(\forall j<\absval{v})d(\nu(\pr_1 u_i),\nu(\pr_1 v_j))>\nu_{\Qplus}(\pr_2 u_i)+\nu_{\Qplus}(\pr_2 v_j).
\end{equation*}
For any ideal cover $u\in\N^{*}$ the \emph{formal diameter} of $u$ is
\begin{equation*}
D\langle u\rangle:=\max_{i,j<\absval{u}}d(\nu(\pr_1 u_i),\nu(\pr_1 u_j))+\nu_{\Qplus}(\pr_2 u_i)+\nu_{\Qplus}(\pr_2 u_j).
\end{equation*}
\end{defi}
Informally, we refer to both $w\in\N^{*}$ and $U\langle w\rangle:=\bigcup_{i<\absval{w}}\alpha(w_i)$ as the ideal cover $w$.
\begin{defi}
Define
$\mathcal{Z}_{\textrm{c}}(X):=\{Y\in\mathcal{K}(X)\setconstr \dim Y\leq 0\}$ and
$\deltadc:\subseteq\Bairespc\to\mathcal{Z}_{\textrm{c}}(X)$ by
\begin{equation*}
\begin{split}
p\in(\deltadc)^{-1}\{K\} :\iff \big\{\langle\langle w^{(0)}\rangle,\dots,\langle w^{(l-1)}\rangle\rangle\setconstr
l\in\N,\: (w^{(i)})_{i<l}\subseteq\N^{*}, \\
\textstyle\bigcup_{i<l}U\langle w^{(i)}\rangle\supseteq K,\: (w^{(i)})_{i<l}\text{ pairwise formally disjoint}
\big\} = \{a\in\N\setconstr (\exists i)p_i=a+1\}.
\end{split}
\end{equation*}
\end{defi}
Informally, $p\in(\deltadc)^{-1}\{K\}$ iff $p$ is a padded list of all formally disjoint tuples of ideal covers
which together cover $K$.  Representation $\deltadc$ will not be used extensively in this paper, but may be of
independent interest.  When considering effective zero-dimensionality of $X$ (as in \pref{sec:zdspcs}), it is
also useful to define a representation of the class $\KO$ of compact open subsets:
\begin{equation*}
\delta_{\KO} :=\delta_{\Delta^0_1(X)}|^{\KO} \sqcap \deltac|^{\KO}.
\end{equation*}

Finally, define
$\hat{D} :\subseteq\Sigma^0_1(X)\times \mathcal{K}_{>}(X)\crsarr
\Delta^0_1(X)^{\N}\times(\Zplus)^{\N}\times\{0,1\}^{\N}$ by declaring
\begin{equation*}
\begin{split}
((W_i)_i,r,s)\in \hat{D}(U,K)\onespace \text{ iff }\onespace
(W_i)_i \text{ pairwise disjoint, } \textstyle\bigcup_i W_i=U\subseteq K,\\
(\forall i)\left( W_i=\emptyset\smash{\iff } s_i=0 \right) \text{ and }
 (\forall n)(\forall j<r_n)\left(\diam W_{\sum_{i<n}r_i+j}<(n+1)^{-1} \right).
\end{split}
\end{equation*}
Here $\mathcal{K}_{>}(X)$ denotes class $\mathcal{K}(X)$ equipped with the representation $\deltac$.
The operation $\hat{D}$ roughly corresponds to the statement of \cite[\S 26.II, Cor 1]{Kuratowski}.
\begin{prop}\label{prn:kzd}
Consider the following conditions on computable metric space $X$:
\begin{enumerate}
\item\label{effzdi} $X$ is $\deltadc$-computable.
\item\label{effzdiiprime} There exist a basis $\mathcal{B}$ for $\mathcal{T}_X$ and computable $b:\N\to\KO$
with $\img b=\mathcal{B}\subseteq\KO$.
\item\label{effzdvi} operation $\hat{D}$ is computable and there exists computable $\gamma:\N\to\Sigma^0_1(X)\times\mathcal{K}_{>}(X)$ such that $\mathcal{B}:=\img(\pr_1\compose\gamma)$ is a basis for $\mathcal{T}_X$ and
$(\forall a\in\N)(\pr_1\gamma(a)\subseteq\pr_2\gamma(a))$.
\item\label{effzdiiagain} There exist computable $b:\N\to\Delta^0_1(X)$ and c.e.~refined inclusion of
$b\mathwrt\alpha$
such that $\mathcal{B}:=\img b$ is a basis for $\mathcal{T}_X$.

\end{enumerate}
Then $(\ref{effzdi})\implies (\ref{effzdiiprime}) \iff (\ref{effzdvi})\implies(\ref{effzdiiagain})$.
\end{prop}
\begin{proof}\hfill

\noindent{(\ref{effzdi})$\implies$(\ref{effzdiiprime}):}
If $p$ is a computable $\deltadc$-name for $X$, for each $n^{\prime}\in\N$ we can compute the
$(n^{\prime})^{\Th}$ tuple $\langle n,k\rangle$ that satisfies $p_n\geq 1$ and $k<\absval{\nu_{\N^*}(p_n-1)}$.
Note $n^{\prime}$ can be arbitrarily large since $\deltadc$ has complete names and any tuple
$\langle w^{(0)}\rangle\dots\langle w^{(l-1)}\rangle$ of formally disjoint ideal covers covering $X$ can be padded by adding copies of the empty cover.  Writing $\langle w^{(n,k)}\rangle:=\nu_{\N^{*}}(p_n-1)_k$ for any
such $n,k$, note
\begin{equation*}
b^{\prime}(n^{\prime})=b\langle n,k\rangle=K :=
\bigcup_{i<\absval{w^{(n,k)}}}\alpha(w^{(n,k)}_i)=\bigcup_{i<\absval{w^{(n,k)}}}\hat{\alpha}(w^{(n,k)}_i)
\end{equation*}
using formal disjointness.  In particular, finite unions preserve openness and closedness properties, while
$K$ is compact as a closed subset of $X$.

We can further compute some $q´\in\delta_{\Sigma^0_1}^{-1}\{K\}$ and $r´\in\delta_{\Pi^0_1}^{-1}\{K\}$.
Clearly $\langle q´,r´\rangle$ is a $\delta_{\Delta^0_1}$-name for $K$, and the definition of
$\deltadc$ ensures $b^{\prime}(n^{\prime})$
runs over a basis for topology of $X$ by the following argument.
Given $\eta>0$, by compactness and zero-dimensionality there exist
finitely many points $(x_k)_{k<l_0}\subseteq X$ and a finite partition $(U_i)_{i<l}\subseteq\Sigma^0_1(X)$
such that $\mathcal{U}_0=(\ball{x_k}{\frac{\eta}{2}})_{k<l_0}$ is a cover of $X$ and $(U_i)_{i<l}$ is a refinement of
$\mathcal{U}_0$.
Each $U_i=X\setminus\bigcup_{i'\neq i}U_{i'}$ is compact with $\diam U_i<\eta$ and we claim we can pick ideal
covers $w^{(i)}\in\N^{*}$ of each $U_i$ ($i<l$) which are pairwise formally disjoint and each have formal
diameter $<\eta$ (this ensures the basis condition is met for `components' $U\langle w^{(i)}\rangle$).  Namely,
let
\begin{equation*}
r:=\min_i d(U_i,X\setminus U_i)=\min_{i,i':i\neq i'}d(U_i,U_{i'}) \:\text{ ($>0$ by compactness)}
\end{equation*}
and $D:=\max_i\diam U_i$ ($<\eta$).  Clearly any respective irredundant ideal covers $w^{(i)}$, $w^{(i')}$
of $U_i$, $U_{i'}$ ($i\neq i'$) with each radius $<\frac{1}{2}\min\{r,\eta-D\}$ satisfy
\begin{equation*}
\nu(\pr_1 w^{(i)}_j)\in U_i\wedge \nu(\pr_1 w^{(i')}_{j'})\in U_{i'}\smash{\implies}
d(\nu(\pr_1 w^{(i)}_j),\nu(\pr_1 w^{(i')}_{j'}))\smash{\geq} r >
\nu_{\Qplus}(\pr_2 w^{(i)}_j)+\nu_{\Qplus}(\pr_2 w^{(i')}_{j'})
\end{equation*}
(for any $j<\absval{w^{(i)}}$, $j'<\absval{w^{(i')}}$) and also
\begin{equation*}
d(\nu(\pr_1 w^{(i)}_j),\nu(\pr_1 w^{(i)}_{j'})) + \nu_{\Qplus}(\pr_2 w^{(i)}_j) +
\nu_{\Qplus}(\pr_2 w^{(i)}_{j'}) < D+(\eta-D)=\eta
\end{equation*}
(for any $j,j'<\absval{w^{(i)}}$).  This completes proof of the claim above.

Finally we observe $b,b^{\prime}:\subseteq\N\to\KO$ are computable (since $p$ computable).  We have written
$b^{\prime}(n^{\prime})=b\langle n,k\rangle$ for convenience, however the domain of $b$ depends on $p$,
whereas $b^{\prime}$ is total.

\noindent{(\ref{effzdiiprime})$\implies$(\ref{effzdiiagain}):} Let $F:\subseteq\Bairespc\to\Bairespc$
be a computable $(\delta_{\N},\deltac)$-realiser of $b$ and define $\sqsubset^{\prime}$ by
\begin{equation*}
c \sqsubset^{\prime} d :\iff \left(
 \text{$F(c.0^{\ordomega})$ enumerates an ideal cover $u$ with $(\forall i<\absval{u})(u_i\sqsubset d)$} \right).
\end{equation*}
Then $(\sqsubset^{\prime})\subseteq\N^2$ is c.e.~and is a formal inclusion of $b$ with respect to $\alpha$
satisfying property (\ref{inclfive}) from \pref{def:incldef}.  In fact,
$\sqsubset^{\prime}$ coincides with set inclusion ($c\sqsubset^{\prime} d$ iff $b(c)\subseteq\alpha(d)$), as we now show.
First, assume $\emptyset\neq b(c)\subseteq\alpha(d)$.  By compactness,
$\tau:=\nu_{\Qplus}(\pr_2 d)-\max_{z\in b(c)}d(z,\nu(\pr_1 d))>0$.  Pick an irredundant ideal cover $u$ of
$b(c)$
such that $u_i\sqsubset d$ for each $i<\absval{u}$.
For instance, consider all $a\in\N$ such that $d_{b(c)}(\nu(\pr_1 a))<\nu_{\Qplus}(\pr_2 a)<\frac{\tau}{2}$
(then take a finite subcover): for appropriate $z\in b(c)$ we have
\begin{align*}
d(\nu(\pr_1 a),\nu(\pr_1 d))+\nu_{\Qplus}(\pr_2 a) &\leq d(\nu(\pr_1 a),z) + d(z,\nu(\pr_1 d)) + \nu_{\Qplus}(\pr_2 a) \\
& < 2\nu_{\Qplus}(\pr_2 a) + (\nu_{\Qplus}(\pr_2 d)-\tau) \leq \nu_{\Qplus}(\pr_2 d),
\end{align*}
so $a\sqsubset d$.  Then $u_i\sqsubset d$ for all $i<\absval{u}$ and $u$ is enumerated in any $\deltac$-name of
$b(c)$, hence
$c\sqsubset^{\prime}d$.  As $u=\emptystring$ is enumerated in any $\deltac$-name of $b(c)=\emptyset$, the same conclusion holds
without assuming $b(c)\neq\emptyset$.

\noindent{(\ref{effzdvi})$\implies$(\ref{effzdiiprime}):}
Let $F$ and $G$ be computable realisers of $\hat{D}$ and $\gamma$ respectively, and write
$(F\compose G)(k.0^{\ordomega})=\langle\langle\langle t^{(0)},\dots\rangle,r\rangle,s\rangle$.  Then $
H:\subseteq\Bairespc\to\Bairespc,\langle j,k\rangle.p\mapsto t^{(j)}$
is computable and we claim $b:\N\to\Delta^0_1(X),\langle j,k\rangle\mapsto
(\delta_{\Delta^0_1}\compose H)(\langle j,k\rangle.0^{\ordomega})$ is a basis numbering.  For, if $x\in X$,
$U\in\mathcal{T}_X$ with $x\in U$ then there exists $k\in\N$
such that $x\in (\pr_1\compose\gamma)(k)=(\pr_1\compose G)(k.0^{\ordomega})\subseteq U$.  Since
$(\hat{D}\compose\gamma)(k)$ is equal to
\begin{equation*}
(\hat{D}\compose[\delta_{\Sigma^0_1},\deltac]\compose G)(k.0^{\ordomega})\ni
([[\delta_{\Delta^0_1}^{\ordomega},\id_{\Bairespc}],\id_{\Bairespc}|^{\{0,1\}^{\N}}]\compose F\compose G)(k.0^{\ordomega}) = ((\delta_{\Delta^0_1}(t^{(i)}))_{i\in\N},r,s)
\end{equation*}
we in particular have
$(\pr_1\compose\gamma)(k)=\bigcup_i\delta_{\Delta^0_1}(t^{(i)})$, so $x\in\delta_{\Delta^0_1}(t^{(j)})=(\delta_{\Delta^0_1}\compose H)(\langle j,k\rangle.0^{\ordomega})$ for some $j\in\N$.

Finally we observe in fact $\img b\subseteq\KO$ with $b:\N\to\KO$ computable.  More formally,
$b\langle j,k\rangle=(\iota\compose b)\langle j,k\rangle\cap(\pr_2\compose\gamma)(k)$ for all $j,k\in\N$ where
$\iota:\Delta^0_1(X)\to\Pi^0_1(X)$ and $\cap:\Pi^0_1(X)\times\mathcal{K}_{>}(X)\to\mathcal{K}_{>}(X)$ are
computable.

\noindent{(\ref{effzdiiprime})$\implies$(\ref{effzdvi}) (Proof sketch):}
Given $p\in\Bairespc=\dom\delta_{\Sigma^0_1(X)}$, a computable realiser $F:\subseteq\Bairespc\to\Bairespc$ of
$b$ and c.e.~formal inclusion $\sqsubset^{\prime}$ as in (\ref{effzdiiagain}), dovetail checking if
$m\sqsubset^{\prime}p_i-1$ (over $m,i\in\N$ such that $p_i\geq 1$).  If so, the computation using index $m$ ends, we
increment $n$ and dovetail output of $(F(m.0^{\ordomega})_{2 k})_{k\in\N}$ as $p^{(n)}$ in
$\langle p^{(0)},p^{(1)},\dots\rangle$.

This describes (without direct use of compactness information from $\delta_{\KO}$) a computable map
$G:\Bairespc\to\Bairespc$ realising
\begin{equation*}
V:\Sigma^0_1(X)\crsarr \Delta^0_1(X)^{\N}, U\mapsto \{(W_i)_i\setconstr \textstyle\bigcup_i W_i=U,
(\forall N)(\exists i\geq N)(\diam W_i<(N+1)^{-1}\}.
\end{equation*}
If $(U,K)\in\dom\hat{D}$ (i.e.~$U\subseteq K$) and $(\tilde{W}_i)_i\in V(U)$,
$W_i^{*}:=\tilde{W}_i\setminus\bigcup_{j<i}\tilde{W}_j$, we can also write
$W_i^{\prime}:=\iota^{\prime}(W_i^{*})\cap K$
where $\cap:\Pi^0_1(X)\times\mathcal{K}_{>}(X)\to\mathcal{K}_{>}(X)$ and
$\iota^{\prime}:\Delta^0_1(X)\to\Pi^0_1(X)$ are computable.  Using compactness, for each $i$ an ideal cover
$w^{(i)}\in\N^{*}$ of $W_i^{\prime}$ can be found, by ideal balls of formal diameter $<(i+1)^{-1}$ and formally
included in $W_i^{*}$.

Considering relatively open sets in $W_i^{*}$, apply the reduction principle to the cover
$\alpha(w^{(i)}_j)\cap W_i^{*}$ ($j<\absval{w^{(i)}}$): let
\begin{equation*}
(W_{i,j})_j \in \tilde{S}^{W_i^{*}}(\alpha(w^{(i)}_0)\cap W_i^{*},\dots,\alpha(w^{(i)}_{\absval{w^{(i)}}-1})\cap W_i^{*},
\emptyset,\emptyset,\dots)\subseteq\Delta^0_1(W_i^{*})\subseteq\Delta^0_1(X).
\end{equation*}
In fact a
$(\delta_{\Sigma^0_1(X)}^{\ordomega},\delta_{\Sigma^0_1(X)}^{\ordomega})$-realiser for $\tilde{S}^X$
will also $(\delta^{\ordomega},\delta^{\ordomega})$-realise $\tilde{S}^Y$ for any $Y\subseteq X$ if $\delta$ is
the representation of $\Sigma^0_1(Y)$ defined from the effective topological space
$(Y,\mathcal{T}_X|_Y,\alpha_Y)$.
A similar statement is true for $R^Y$, so each
$R^{W_i^{*}}\compose\tilde{S}^{W_i^{*}}:\subseteq\Sigma^0_1(W_i^{*})^{\N}\crsarr\Delta^0_1(W_i^{*})^{\N}$ is computable, uniformly in $i$, as
are the inclusions $\Delta^0_1(W_i^{*})\to\Delta^0_1(X)$ (use $\delta_{\Delta^0_1(X)}$-names of $W_i^{*}$ and computability of binary intersection on $\Sigma^0_1(X),\Pi^0_1(X)$ respectively).

Letting $r_i:=\absval{w^{(i)}}$, $W_{\sum_{i<k}r_i+j}:=W_{k,j}$ ($j<\absval{w^{(k)}}$, $k\in\N$), we have sequences $r$, $(W_i)_i$ almost as in definition of $\hat{D}$.  To prove $\hat{D}$ computable it remains to ensure $r_i\geq 1$ for all $i$ and detect nonemptiness of the $W_i$.
From a $\delta_{\Delta^0_1(X)}$-name of $W_{i,j}$ and $\deltac$-name of $K$ ($\supseteq U\supseteq W_{i,j}$), a
$\delta_{\KO}$-name of $W_{i,j}$ is computable.  Also, $z|^{\KO}$ is $(\delta_{\KO},\delta_{\N}|^{\{0,1\}})$ computable where
\begin{equation*}
z:\Delta^0_1(X)\to\{0,1\},W\mapsto \begin{cases}
0, &\mbox{\textrm{ if $W=\emptyset$}}\\
1, &\mbox{\textrm{ if $W\neq\emptyset$}}\end{cases}
\end{equation*}\enlargethispage{\baselineskip}
Fixing some $a_0\in\dom\alpha=\N$ we modify the above argument to pick $w^{(i)}$ as a one-element cover $a_0\in\N\subseteq\N^{*}$
if $W_i^{*}=\emptyset$, and choose $w^{(i)}$ irredundant otherwise (nonemptiness of $\alpha(w^{(i)}_j)=\alpha(w^{(i)}_j)\cap W_i^{*}$
is clearly decidable without using $z$).  Then $r_i\geq 1$ for all $i$.

This completes the proof.
\end{proof}

As an application of \pref{prn:kzd} (using the operation $\hat{D}$), we present an effectivisation
\pref{thm:retracta} of (\ref{zdix}),
the retract characterisation of zero-dimensionality from \pref{sec:intro}.  In \pref{sec:retractsec}
a converse to this result will be proven.
Before stating the theorem, we give two lemmas relevant for dealing with compactness in situations
involving the representations $\delta_{\textrm{dist}}^{>}$, $\deltar$.
For any closed $A,B\subseteq X$, denote $d(A,B):=\inf\{d(x,y)\setconstr x\in A, y\in B\}$ with the convention
$\inf\emptyset=\infty$.
\begin{lem}\label{lem:infdist}
For any computable metric space $X$, $\hat{d}:\subseteq\mathcal{A}(X)\times\mathcal{K}(X)\to\overline{\R},
(A,K)\mapsto d(A,K)$ ($\dom\hat{d}=\{(A,K)\setconstr A\neq\emptyset\}$) is
$([\delta_{\textrm{dist}}^{>},\deltac],\overline{\rho_{<}})$-computable.
\end{lem}
\begin{proof}
Suppose $p\in(\delta_{\textrm{dist}}^{>})^{-1}\{A\}$, $q\in\deltac^{-1}\{K\}$, $r\in\Q$.  Then we claim
\begin{align*}
d(A,K)>r & \iff (\exists n)(\exists w\in\N^{*})\left( q_n=\langle w\rangle\wedge
(\forall i<\absval{w})\left(
 d_A(\nu(\pr_1 w_i))-\nu_{\Qplus}(\pr_2 w_i) > r \right) \right)\\
 & \iff (\exists n)(\exists w\in\N^{*})\bigl( q_n=\langle w\rangle \wedge \\
& (\forall i<\absval{w})(\exists j,k)
\left( (\eta_p\compose F)(\pr_1 w_i.0^{\ordomega})_j=k+1\wedge \nu_{\Q}(k)>r+\nu_{\Qplus}(\pr_2 w_i) \right)
\bigr)
\end{align*}
where $F:\subseteq\Bairespc\to\Bairespc$ is a computable $(\delta_{\N},\delta_X)$-realiser of $\nu:\N\to X$.

For the first equivalence, if $d(A,K)>r$ then every $x\in K$ has $d_A(x)>r$ and by density of $\nu$ and continuity of $d_A$ there exists $a\in\N$ such that $x\in\alpha(a)$ and
$(d_A\compose\nu)(\pr_1 a)>r+\nu_{\Qplus}(\pr_2 a)$.  Compactness gives an ideal cover
$w$ as required.  Conversely, given such $w$, any $x\in K$ has some $i<\absval{w}$ such that $x\in\alpha(w_i)$, so
$d_A(x)>(d_A\compose\nu)(\pr_1 w_i)-\nu_{\Qplus}(\pr_2 w_i)$.  Now $d(A,K)=\inf_{x\in K}d_A(x)\geq
\min_{i<\absval{w}}\left( (d_A\compose\nu)(\pr_1 w_i)-\nu_{\Qplus}(\pr_2 w_i) \right) > r$.  One checks this
argument works for $K=\emptyset$ also.
The second equivalence follows from $p\in (\delta_{\textrm{dist}}^{>})^{-1}\{A\}$.
\end{proof}

\begin{lem}\label{lem:coverlem}
Let $X$ be a computable metric space.  If $K\subseteq X$ is compact and $K\subseteq\nbhd{\cl{S}}{\epsilon}$ then
there exist $(s_i)_{i<n}\subseteq S$ and an ideal cover $v\in\N^{*}$ of $K$ such that $v$ `formally refines'
$(\ball{s_i}{\epsilon})_{i<n}$,
i.e.~for every $i<\absval{v}$ there exists $j<n$
such that $d(\nu(\pr_1 v_i),s_j)+\nu_{\Qplus}(\pr_2 v_i)<\epsilon$.
\end{lem}
\begin{proof}
Whenever $x\in\ball{s}{\epsilon}$ we can pick $q\in\Qplus$ with $d(x,s)+q<\epsilon$, then $a\in
\nu^{-1}( \ball{x}{q}\cap\ball{s}{\epsilon-q} )$ (so $b=\langle a,\overline{q}\rangle$
satisfies $x\in\alpha(b)$ and $d(\nu(\pr_1 b),s)+\nu_{\Qplus}(\pr_2 b)<\epsilon$).  But applying compactness once gives $K\subseteq\bigcup_{i<n}\ball{s_i}{\epsilon}$ for some $(s_i)_{i<n}\subseteq S$, and again gives an
ideal cover as desired.
\end{proof}

\begin{thm}(cf.~\cite[Cor 26.II.2]{Kuratowski})\label{thm:retracta}
Suppose $X$ is $\deltadc$-computable.  Then
$E:\subseteq\mathcal{A}(X)\crsarr C(X,X),A\mapsto \{f\setconstr\img f=A\wedge f|_A=\id_A\}$
($\dom E=\mathcal{A}(X)\setminus\{\emptyset\}$) is well-defined and computable (where $\mathcal{A}(X)$ is represented by $\deltar\sqcap\delta_{\textrm{dist}}^{>}$).
\end{thm}
\begin{proof}[Proof Sketch]
First (by \pref{prn:kzd}) recall $\hat{D}:\subseteq\Sigma^0_1(X)\times\mathcal{K}_{>}(X)\crsarr
\Delta^0_1(X)^{\N}\times(\Zplus)^{\N}\times\{0,1\}^{\N}$ is computable, say let $G$ be a computable realiser.
For a fixed name of $A\in\mathcal{A}(X)\setminus\{\emptyset\}$ as input, consider corresponding
$((W_i)_i,\xi,s)\in\hat{D}(X\setminus A,X)$ and pick
$(x_i\in W_i, \text{ if $s_i\neq 0$; } x_i\in\img\nu, \text{ if $s_i=0$})$
and also $y_i\in A$ computably such that $d(x_i,y_i)<d(A,W_i)+(i+1)^{-1}$ ($i\in\N$).  This is possible since
$d(A,W_i)$ is computable from below
uniformly in the input and $i$ (use \pref{lem:infdist}), and since $\deltar$-names of $A$, $W_i$ are available.

Next define
\begin{equation*}
f:X\to X,x\mapsto\begin{cases}
x, &\mbox{\textrm{ if $x\in A$}}\\
y_i, &\mbox{\textrm{ if $(\exists i)W_i\ni x$}}.
\end{cases}
\end{equation*}
That $f$ is continuous is shown by Kuratowski; we will check $f$ is computable in the inputs directly by
showing $f^{-1}V$ is computable uniformly in the inputs and a $\delta_{\Sigma^0_1}$-name of
$V\in\Sigma^0_1(X)$.  Roughly speaking, we consider (instead of disjoint cases as in the definition of $f$) a
disjunction $(\exists i)x\in W_i\vee (\exists N)(x\in\bigcap_{i<N}(X\setminus W_i))$ where in the second case
$N$ has to be suitably large.  This will be used
to define computable $F:\subseteq\Bairespc\to\Bairespc,\langle\langle p,q\rangle,r\rangle\mapsto t$ so that
each induced function $u=\delta_{\Sigma^0_1}\compose F\langle\langle\cdot,q\rangle,r\rangle:\dom\delta_X \to
\Sigma^0_1(X)$ satisfies ($\delta_X(p)\in u(p)\subseteq f^{-1}\delta_{\Sigma^0_1}(r)$, if
$p\in\delta_X^{-1}f^{-1}\delta_{\Sigma^0_1}(r)$;
$u(p)=\emptyset$, if $p\in\dom\delta_X\setminus\delta_X^{-1}f^{-1}\delta_{\Sigma^0_1}(r)$).  Then \pref{lem:ctslindelof}
can be applied to computably obtain a name for $f^{-1}\delta_{\Sigma^0_1}(r)$ ($f$ being dependent on
$q=\langle\langle q^{(0)},\dots\rangle,\xi,s,\langle t^{(0)},\dots\rangle\rangle$ where $q^{(i)}\in\delta_{\Delta^0_1}^{-1}\{W_i\}$, $t^{(i)}\in\delta_X^{-1}\{y_i\}$ for all $i\in\N$).

To define $t$, we
dovetail repeated output of `$0$' with searching for large $M,N$ and an ideal ball
$a=\langle p_j,\overline{2^{-j+1}}\rangle$ small enough to satisfy
\begin{equation*}
R_0(p,q,r,j) :\equiv
(\exists i,k,l,m)\left( q^{(i)}_{2 k}\geq 1\wedge a\sqsubset q^{(i)}_{2 k}-1\wedge r_l\geq 1\wedge
\langle t^{(i)}_m,\overline{2^{-m+1}}\rangle\sqsubset r_l-1 \right)
\end{equation*}
or
\begin{multline*}
R_1(p,q,r,j,M,N) :\equiv (\forall i<N)(\exists k)\left(
 q^{(i)}_{2 k+1}\geq 1\wedge a\sqsubset q^{(i)}_{2 k+1}-1 \right) \wedge
N\geq \sum_{i<M}\xi_i\wedge \\
2^{-j+2}\leq (M+1)^{-1}\wedge
(\exists v\in\N^{*})\big( \text{ $v$ appears in a $\deltac$-name for $\textstyle\bigcap_{i<N}(X\setminus W_i)$,}
\\
\text{$v$ `formally refines' a finite cover by $(M+1)^{-1}$-balls about points of the $\deltar$-name of $A$} \big) \\
\wedge
(\exists k)\left( r_k\geq 1\wedge\left\langle p_j,\overline{2^{-j+1}+\frac{3}{M+1}}\right\rangle
\sqsubset r_k-1\right)
\end{multline*}

If found we should output `$a+1$' followed by $0^{\ordomega}$.  Any $x\in f^{-1}V$ either has
$W_i\ni x$ for some $i$ or else $x\in A$.  In the former case, $y_i\in V$ and property (\ref{inclone}) from
\pref{def:incldef} applied
twice gives $R_0(p,q,r,j)$, so assume $x\in A$.  From $V\ni f x=x$ and $r\in\delta_{\Sigma^0_1}^{-1}\{V\}$ we
can pick $k,M$ such that $r_k\geq 1$ and $d(x,\nu(\pr_1(r_k-1)))+\frac{4}{M+1}<\nu_{\Qplus}(\pr_2(r_k-1))$,
then $N\geq\sum_{i<M}\xi_i$ ($\geq M$) such that $\bigcap_{i<N}(X\setminus W_i)\subseteq\nbhd{A}{(M+1)^{-1}}$, then
finally $j\in\N$, $v\in\N^{*}$ and $a=\langle p_j,\overline{2^{-j+1}}\rangle$ as follows: such that $v$ appears in a
$\deltac$-name for $\bigcap_{i<N}(X\setminus W_i)$, $v$ `formally refines' $(\ball{z_i}{(M+1)^{-1}})_{i<n}$ for
some $z_0,\dots,z_{n-1}\in A$ given by the input $\deltar$-name of $A$,
$a$ is `formally included' in $\bigcap_{i<N}(X\setminus W_i)$, $2^{-j+2}\leq (M+1)^{-1}$ and
$\langle p_j,\overline{2^{-j+1}+\frac{3}{M+1}}\rangle\sqsubset r_k-1$.  To see such $j$ exists use $\nu(p_j)\to x$, $2^{-j+1}\to 0$ and continuity of $d$ in inequalities corresponding to the last three requirements; to
see suitable $v$ exists use \pref{lem:coverlem}.
In this case one checks $R_1(p,q,r,j,M,N)$ holds.

Conversely, we will show $U:=\delta_{\Sigma^0_1}(t)$ must be contained in $f^{-1}V$,
indeed that any $j,M,N$ with $R_0(p,q,r,j)\vee R_1(p,q,r,j,M,N)$ must correspondingly satisfy
$\alpha(a)\subseteq f^{-1}V$ (where $a=\langle p_j,\overline{2^{-j+1}}\rangle$).  For, in the first
clause necessarily $f\alpha(a)\subseteq f(W_i)=\{y_i\}\subseteq V$, so we suppose the second clause holds.
Now any $z\in\alpha(a)$ has either $z\in A$ or $W_i\ni z$ for some $i$.  In the first case,
$f z=z\in\alpha(a)\subseteq\ball{\nu(p_j)}{2^{-j+1}+\frac{3}{M+1}}\subseteq V$.  In the second case,
$z\in\alpha(a)$ implies $i\geq N$ and $d_A(z)<(M+1)^{-1}$, so
\begin{align*}
d(f z,\nu(p_j)) & =d(y_i,\nu(p_j))\leq d_{W_i}(y_i)+\diam W_i+d(z,\nu(p_j)) \\
& <(d(A,W_i)+(i+1)^{-1})+\diam W_i+2^{-j+1}<\frac{3}{M+1}+2^{-j+1}.
\end{align*}
This completes the proof.
\end{proof}

\section{Bilocated subsets}\label{sec:retractsec}

In this final section, we present a converse to \pref{thm:retracta} (in other words, an effectivisation of
the reverse direction of \cite[Thm 7.3]{Kechris}), namely \pref{prn:retractb}.  This relies on a version of the
construction of so-called bilocated sets from the constructive analysis literature --- see \pref{prn:bilocated}.
Such a construction for us involves an application of the effective Baire category theorem and a decomposition
of compact sets formally different to that in \pref{sec:repnsec} (see \pref{thm:decomp}).  The proofs of both
\pref{thm:decomp} and \pref{prn:bilocated} are adapted to computable analysis in an ad hoc way (not following an
established interpretation of constructive proofs in this context).  It is also worth noting a constructive development of dimension theory exists \cite{RBCM76}, \cite{BJMR77} which, though based on information weaker
than we shall consider, does also use bilocated subsets fundamentally \cite[Thm 0.1]{BJMR77}.

We begin with several representations from \cite{BrattkaPresser}, namely with $\deltamc$ (similar to
$\deltac$ except that each ball of each ideal cover is required to intersect $K$), $\deltar'$ and
$\delta_{\textrm{Hausdorff}}$.  Here
$\langle q,p^{(0)},p^{(1)},\dots\rangle\in(\deltar')^{-1}\{K\}$ iff
$\{p^{(i)}\setconstr i\in\N\}\subseteq\dom\delta_X$,
$K=\clvar\{\delta_X(p^{(i)})\setconstr i\in\N\}$, $q$ is unbounded and
$\hausdmetr(K_i,K_j)<2^{-\min\{i,j\}}$ for all $i,j\in\N$,
where $K_i:=\{\delta_X(p^{(k)})\setconstr k\leq q_i\}$ ($i\in\N$).

To define $\delta_{\textrm{Hausdorff}}$, we first consider $\mathcal{K}(X)\setminus\{\emptyset\}$ metrized by the Hausdorff metric $\hausdmetr$ and denote
$\mathcal{Q}:=\{A\subseteq\img\nu\setconstr A\text{ finite, }A\neq\emptyset\}=E(\img\nu)\setminus\{\emptyset\}
\subseteq \mathcal{K}(X)\setminus\{\emptyset\}$
with numbering $\nu_{\mathcal{Q}}$ defined by
$\nu_{\mathcal{Q}}\langle w\rangle:=\{\nu(w_i)\setconstr i<\absval{w}\}$ for any
$w\in\N^{*}\setminus\{\emptystring\}$.  Then $p\in\delta_{\textrm{Hausdorff}}^{-1}\{K\}$ iff
$\img p\subseteq\dom\nu_{\mathcal{Q}}$, $\hausdmetr(\nu_{\mathcal{Q}}(p_i),\nu_{\mathcal{Q}}(p_j)) <
2^{-\min\{i,j\}}$ for all $i,j\in\N$ and $K=\lim_{i\to\infty}\nu_{\mathcal{Q}}(p_i)$ with respect to
$\hausdmetr$.  Most relevant below will be the following result from \cite[Thm 4.12]{BrattkaPresser}:
\begin{lem}
$\delta_{\textrm{Hausdorff}}\equiv\delta_{\textrm{range}}'\equiv\deltamc|^{\mathcal{K}(X)\setminus\{\emptyset\}}$.
\end{lem}

On the decomposition of arbitrary compact sets, we then have the following result (a version of
\cite[Thm (4.8)]{BishopBridges}).
\begin{thm}\label{thm:decomp}
Let $X$ be a computable metric space.  Then $S:\mathcal{K}(X)\times\N\crsarr\mathcal{K}(X)^{*}$ defined by
\begin{equation*}
S(K,l):=\{K_0\dots K_{n-1}\setconstr K=\bigcup_{i<n}K_i,\:\max_{i<n}\diam K_i<2^{-l}\}
\end{equation*}
is $([\deltamc,\delta_{\N}],\deltamc^{*})$-computable.
\end{thm}
\begin{proof}[Proof sketch]
Assume $F:\subseteq\Bairespc\to\Bairespc$ is a computable witness of
$\deltamc|^{\mathcal{K}(X)\setminus\{\emptyset\}}\leq\delta_{\textrm{range}}'$.
Given $p\in\deltamc^{-1}\{K\}$ and $l$, compute $n\in\N$ and singletons $X^0_j\subseteq K$ ($j<n$) such that
$(\forall x\in K)(\min_{j<n}d_{X^0_j}(x)<3^{-2}2^{-l})$; for instance, use appropriately the $(l+4)^{\Th}$
finite approximation to $K$ from the $\delta_{\textrm{range}}'$-name $F(p)$.  Similarly, for each $i\in\N$, we
define $X^{i+1}_j$ ($j<n$) in terms of corresponding $X^i_j$ as follows: find a strict
$3^{-i-3}2^{-l}$-approximation $\{x_j\setconstr j<N\}\subseteq K$ to
$K$ (using $F(p)$ appropriately); compute some partition $S\disjunion T=[0,N)\times [0,n)\subseteq\N^2$ where
\begin{equation*}
(m,j)\in S\implies d_{X^i_j}(x_m)<3^{-i-1}2^{-l} \onespace\text{ and }\onespace
(m,j)\in T\implies d_{X^i_j}(x_m)>3^{-i-1}2^{-l-1};
\end{equation*}
for $j<n$ let $X^{i+1}_j:=X^i_j\union\{x_m\setconstr m<N,\: (m,j)\in S\}$.  The finite sets $X^i_j$ ($j<n$,
$i\in\N$) thus defined easily satisfy the first two properties of
\begin{enumerate}
\item\label{decompone} $X^i_j\subseteq X^{i+1}_j$,
\item\label{decomptwo} $(\forall x\in K)(\forall j<n)\left( x\in X^{i+1}_j\implies d_{X^i_j}(x)<3^{-i-1}2^{-l} \right)$,
\item\label{decompthree} $(\forall x\in K)(\forall j<n)\left( d_{X^i_j}(x)<3^{-i-2}2^{-l}\implies d_{X^{i+1}_j}(x)<3^{-i-3}2^{-l} \right)$.
\end{enumerate}
For the third, let $x\in K$ with $d_{X^i_j}(x)<3^{-i-2}2^{-l}$ and choose $m<N$ such that
$d(x,x_m)<3^{-i-3}2^{-l}$.
We have
\begin{align*}
& d_{X^i_j}(x_m)\leq d(x_m,x)+d_{X^i_j}(x)<(3^{-i-3}+3^{-i-2})2^{-l}<3^{-i-1}2^{-l-1} \\
& \implies (m,j)\not\in T \implies (m,j)\in S\implies x_m\in X^{i+1}_j.
\end{align*}
It follows that $d_{X^{i+1}_j}(x)\leq d(x,x_m)<3^{-i-3}2^{-l}$.

We now check $Y_j:=\bigcup_{i\in\N}X^i_j$, or rather $X_j:=\cl{Y_j}$ ($j<n$) satisfy total boundedness and the diameter condition.  First, consider
$m\in\N$ and $y\in Y_j$.  For $i$ with $y\in X^i_j$, either $i\leq m$ (so $d_{X^m_j}(y)=0$) or $i>m$.  In the latter case,
$(\ref{decomptwo})|_{k\in [m,i)}$ allows to construct $(y_k)_{k=m}^i$ with
\begin{equation*}
y_i=y\wedge (\forall k)(m\leq k<i\implies
y_k\in X^k_j\wedge d(y_k,y_{k+1})<3^{-k-1}2^{-l})
\end{equation*}
(that is, if $y_{k+1}\in X^{k+1}_j$, pick $y_k\in X^k_j$ such that
$d(y_{k+1},y_k)<3^{-k-1}2^{-l}$, inductively for $k=i-1,\dots,m$).

Then
$d_{X^m_j}(y)\leq d(y_i,y_m)\leq \sum_{m\leq k<i}d(y_k,y_{k+1})<\sum_{k\geq m}3^{-k-1}2^{-l}=
\frac{3^{-m-1}2^{-l}}{1-3^{-1}} =
2^{-l-1}3^{-m}$.
As $X^m_j$ is a finite $3^{-m}2^{-l-1}$-approximation to $Y_j$, it is also a finite
$3^{-m}2^{-l}$-approximation to $X_j$.  Since
$m$ was arbitrary, $X_j$ is totally bounded.  Next consider $i\in\N$ and $x,x'\in X^{i+1}_j$; we have
\begin{equation*}
d(x,x')\leq
d_{X^i_j}(x)+\diam X^i_j+d_{X^i_j}(x') < 3^{-i-1}2^{-l+1}+\sum_{k=1}^i 3^{-k}2^{-l+1} = \sum_{k=1}^{i+1}3^{-k}2^{-l+1}
\end{equation*}
provided
$\diam X^i_j\leq \sum_{k=1}^i 3^{-k}2^{-l+1}$.  Plainly the latter condition holds for $i=0$, so an inductive argument applies.  In particular,
$\diam X_j\leq \sum_{k=1}^{\infty}3^{-k}2^{-l+1}=\frac{3^{-1}2^{-l+1}}{1-3^{-1}}=2^{-l}$.  Finally, if $x\in K$, pick $j<n$ such that
$d_{X^0_j}(x)<3^{-2}2^{-l}$.  By induction on $i\in\N$ using (\ref{decompthree}) we have $d_{X^i_j}(x)<3^{-i-2}2^{-l}$ for all $i$ (case $i=0$ by
choice of $j$).  Thus $Y_j=\bigcup_i X^i_j$ contains points arbitrarily close to $x$, i.e.~$x\in\cl{Y_j}=X_j$.

Using the above construction, observe (\ref{decompone}), (\ref{decomptwo}) imply $\hausdmetr(X^i_j,X^{i+1}_j)\leq 3^{-i-1}2^{-l}$, so if $i'\geq i$ then
\begin{equation*}
\hausdmetr(X^i_j,X^{i'}_j)\leq\sum_{i\leq k<i'}\hausdmetr(X^k_j,X^{k+1}_j) \leq \sum_{i\leq k<i'}3^{-k-1}2^{-l}
< \frac{3^{-i-1}2^{-l}}{1-3^{-1}} = 3^{-i}2^{-l-1}\leq 2^{-i}.
\end{equation*}
Clearly then $\hausdmetr(X^i_j,X^{i'}_j)<2^{-\min\{i,i'\}}$ for all $i,i'$.
Defining $q$ in the obvious way, we obtain a $\delta_{\textrm{range}}'$-name for each $X_j$ ($j<n$), and each can be translated into a
$\deltamc$-name $q^{(j)}$.

Now, consider the possibility that $K=\emptyset$.
Observe for $p\in\dom\deltamc$ that
\begin{equation*}
(p \text{ contains ideal cover }\emptystring) \text{ iff }(p \text{ contains only ideal cover }\emptystring)
\text{ iff }p\in\deltamc^{-1}\{\emptyset\}.
\end{equation*}
Using this condition it is possible to decide from $p\in\deltamc^{-1}\{K\}$ whether $K=\emptyset$.  If
so, we output $n.\langle p,p,\dots\rangle$ for some fixed $n\geq 1$, otherwise we output
$n.\langle q^{(0)},\dots,q^{(n-1)},0^{\ordomega},0^{\ordomega},\dots\rangle$ defined as above.
\end{proof}

Next we must recall the effective Baire category theorem (see \cite[Thm 7.20]{BrattkaHertlingWeihrauch08} and
the references there).
\begin{thm}\label{thm:effbct}
Suppose $X$ is a complete computable metric space.  Then
\begin{equation*}
B:\subseteq\Pi^0_1(X)^{\N}\crsarr X^{\N},
(A_i)_{i\in\N}\mapsto\{(x_i)_{i\in\N}\setconstr (x_i)_i\textrm{ dense in }X\setminus\bigcup_i A_i\}
\end{equation*}
($\dom B=\{(A_i)_i\setconstr \textrm{each $A_i$ nowhere dense}\}$) is computable.
\end{thm}

\begin{prop}(cf.~\cite[Ch 4,Thm 8]{Bishop}, \cite[Ch 7,Prop 4.14]{TroelstravanD})\label{prn:bilocated}
Let $X$ be a computable metric space.
Define $p^{+},p^{-}:\R\times\Cont(X,\R)\to\Sigma^0_1(X)$ and
$P^{+},P^{-}:\mathcal{K}(X)\times\R\times\Cont(X,\R)\to\mathcal{K}(X)$ by
\begin{align*}
p^{+}(\alpha,f) &:=f^{-1}(\alpha,\infty),\onespace p^{-}(\alpha,f):=f^{-1}(-\infty,\alpha),\\
P^{+}(K,\alpha,f) &:=K\cap f^{-1}[\alpha,\infty),\onespace P^{-}(K,\alpha,f):=K\cap f^{-1}(-\infty,\alpha].
\end{align*}
$p^{+},p^{-}$ are computable and $P^{+},P^{-}$ are $(\deltac,\rho,[\delta_X\to\rho];\deltac)$-computable.

Moreover $A$ is $(\deltamc,\delta_{\Q}^2,\delta;[\rho,\deltamc^2])$-computable where
$\delta:=[\delta_X\to\rho]$ and
$A:\subseteq\mathcal{K}(X)\times\Q^2\times\Cont(X,\R)\crsarr\R\times\mathcal{K}(X)^2$ is defined by
\begin{equation*}
\begin{split}
A(K,a,b,f):=\{(\alpha,P^{-}(K,\alpha,f),P^{+}(K,\alpha,f))\setconstr a<\alpha<b \wedge
\cl{K\cap p^{-}(\alpha,f)}=P^{-}(K,\alpha,f) \wedge \\
\cl{K\cap p^{+}(\alpha,f)}=P^{+}(K,\alpha,f)\},
\end{split}
\end{equation*}
with $\dom A=\{(K,a,b,f)\setconstr K\neq\emptyset\wedge a<b\}$.
\end{prop}
\begin{proof}
First, $p^{+},p^{-},P^{+},P^{-}$ are computable: by computability of preimages of $f$
($\Sigma^0_1(\R)\times\Cont(X,\R)\to\Sigma^0_1(X)$, $\Pi^0_1(\R)\times\Cont(X,\R)\to\Pi^0_1(X)$)
and the operation $\cap:\Pi^0_1(X)\times\mathcal{K}_{>}(X)\to\mathcal{K}_{>}(X)$.

Now, given a $\deltamc$-name of $K$, for each $k\in\N$ consider a decomposition $K=\bigcup_{j<N_k}X^k_j$ as in \pref{thm:decomp} with
$\max_j\diam X^k_j<2^{-k-1}$.  We can compute maxima and minima of $f$ on each $X^k_j$, effectively in $k$, $j$
(and uniformly in names of $K$, $f$), and will call these $c^{\pm}_{k,j}$.
As they form a sequence computable from $K$, $f$ (for instance
$c^{-}_{0,0},c^{+}_{0,0};\dots; c^{-}_{0,N_0-1},
c^{+}_{0,N_0-1}; c^{-}_{1,0}, c^{+}_{1,0};\dots$)
one can compute $\alpha\in (a,b)$ which avoids all $c^{\pm}_{k,j}$ (formally, use \pref{thm:effbct}).

Using positive information on $X^k_j$, if
$c^{-}_{k,j}<\alpha$ we compute some $x^{-}_{k,j}\in X^k_j$ with $f(x^{-}_{k,j})<\alpha$, similarly if
$\alpha<c^{+}_{k,j}$ we compute some $x^{+}_{k,j}\in X^k_j$ with $\alpha<f(x^{+}_{k,j})$.
We will write $M^{-}_k:=\{j<N_k\setconstr c^{-}_{k,j}<\alpha\}$,
$M^{+}_k:=\{j<N_k\setconstr c^{+}_{k,j}>\alpha\}$ and set
$Y^{-}_k:=\{x^{\sigma}_{k,j}\setconstr j\in M^{\sigma}_k\}$ ($\sigma=+,-$).
Note that $Y^{\sigma}_k$ is a finite $2^{-k-1}$-approximation to
$X_{\alpha}^{\sigma}:=K\cap p^{\sigma}(\alpha,f)$ (for, $Y^{-}_k\subseteq X^{-}_{\alpha}$ while any
$x\in X^{-}_{\alpha}$ has some $j<N_k$ such that $X^k_j\ni x$, with necessarily $c^{-}_{k,j}\leq f(x)<\alpha$
and $d(x,x^{-}_{k,j})\leq\diam X^k_j<2^{-k-1}$.  The proof for $\sigma=+$ is similar).
As a consequence, we have the equivalence $X^{\sigma}_{\alpha}=\emptyset$ iff
$Y^{\sigma}_k=\emptyset$ for all $k$ iff $Y^{\sigma}_k=\emptyset$ for some $k$, and also for each $k$ the equivalence $Y^{\sigma}_k=\emptyset$ iff $M^{\sigma}_k=\emptyset$.  Moreover, we can (if
$M^{\sigma}_{k+1}\neq\emptyset$) compute finite sets of ideal points approximating
$Y^{\sigma}_{k+1}$: writing $M^{\sigma}_{k+1}=\{j_1,\dots,j_P\}$ in strictly ascending order and
$p^{(k,i)}\in\delta_X^{-1}\{x^{\sigma}_{k+1,j_i}\}$ for the Cauchy name calculated by our algorithm, define
$u^{\sigma}_k=\langle p^{(k,1)}_{k+1},\dots,p^{(k,P)}_{k+1}\rangle$; then
$\nu_{\mathcal{Q}}(u^{\sigma}_k)\subseteq\clnbhd{Y^{\sigma}_{k+1}}{2^{-k-1}}$ and
$Y^{\sigma}_{k+1}\subseteq \clnbhd{\nu_{\mathcal{Q}}(u^{\sigma}_k)}{2^{-k-1}}$.

Overall we get
\begin{equation}\label{eq:finapprox}
\nu_{\mathcal{Q}}(u^{\sigma}_k)\subseteq \clnbhd{X^{\sigma}_{\alpha}}{2^{-k-1}}\wedge
X^{\sigma}_{\alpha}\subseteq \nbhd{\nu_{\mathcal{Q}}(u^{\sigma}_k)}{2^{-k-2}+2^{-k-1}}
\end{equation}
and in particular any $k,l\in\N$ satisfy $\nu_{\mathcal{Q}}(u^{\sigma}_k)\subseteq
\nbhd{\nu_{\mathcal{Q}}(u^{\sigma}_l)}{2^{-k-1}+2^{-l-2}+2^{-l-1}}$.  For $k>l$ thus
\begin{align*}
\hausdmetr(\nu_{\mathcal{Q}}(u^{\sigma}_k),\nu_{\mathcal{Q}}(u^{\sigma}_l)) & <
\max\{2^{-k-1}+2^{-l-2}+2^{-l-1},2^{-l-1}+2^{-k-2}+2^{-k-1}\} \\
& = 2^{-k-1}+2^{-l-2}+2^{-l-1} =
2^{-l}(2^{-(k-l)-1}+2^{-2}+2^{-1})\leq 2^{-l}.
\end{align*}
On the other hand, for any $\epsilon>0$ there exists $k$ with
$2^{-k-2}+2^{-k-1}<\epsilon$ and in this case (\ref{eq:finapprox}) implies
$\hausdmetr(\cl{X^{\sigma}_{\alpha}},\nu_{\mathcal{Q}}(u^{\sigma}_k))<\epsilon$.  For
$X^{\sigma}_{\alpha}\neq\emptyset$ we have thus shown $u^{\sigma}_0 u^{\sigma}_1\dots\in\Bairespc$ is a
$\delta_{\textrm{Hausdorff}}$-name for $\cl{X^{\sigma}_{\alpha}}$, computable from the inputs.

We now verify $\cl{X^{-}_{\alpha}}=K\cap f^{-1}(-\infty,\alpha]$.  First,
$\cl{X^{-}_{\alpha}}\subseteq K\cap f^{-1}(-\infty,\alpha]$ by closedness of $K$ and continuity of $f$.  On the
other hand, suppose there exists $x\in (K\cap f^{-1}\{\alpha\})\setminus\cl{X^{-}_{\alpha}}$, say
$V\in\mathcal{T}_X$ is such that $x\in V\cap K\subseteq K\cap f^{-1}[\alpha,\infty)$.  Then for $k$ sufficiently
large and $j<N_k$ such that
$X^k_j\ni x$, we have $X^k_j\subseteq V\cap K$, but by construction
$c^{-}_{k,j}\leq f(x)=\alpha$ implies $c^{-}_{k,j}<\alpha$ and thus $X^k_j\cap f^{-1}(-\infty,\alpha) \neq
\emptyset$, a contradiction.  $\cl{X^{+}_{\alpha}}=K\cap f^{-1}[\alpha,\infty)$ is verified in a similar way.

Finally, we describe the output of the algorithm.  If $X^{\sigma}_{\alpha}=\emptyset$
(equivalently, $M^{\sigma}_0=\emptyset$) for some $\sigma\in\{\pm\}$ we should output some fixed computable
$\deltamc$-name of $\emptyset$ as the name of $\cl{X^{\sigma}_{\alpha}}$ for the corresponding $\sigma$.
Otherwise, we should compute a $\delta_{\textrm{Hausdorff}}$-name of $\cl{X^{\sigma}_{\alpha}}$ (as above) and
translate this into a $\deltamc$-name.  Since $M^{\sigma}_0\in E(\N)$ is computable from the inputs, the choice between these two cases is decidable.
This completes the description of the algorithm.
\end{proof}

Finally, we give our converse to \pref{thm:retracta}.
\begin{prop}\label{prn:retractb}
Suppose $X$ is $\deltac$-computable and
$E:\subseteq\mathcal{A}(X)\crsarr C(X,X),A\mapsto \{f\setconstr\img f=A\wedge f|_A=\id_A\}$
($\dom E=\mathcal{A}(X)\setminus\{\emptyset\}$) is well-defined and computable, where $\mathcal{A}(X)$ is represented by $\deltar\sqcap\delta_{\textrm{dist}}^{>}$.  Then $X$ is zero-dimensional and
$M$ from \pref{sec:zdspcs} is computable.
\end{prop}
\begin{proof}
First, (nonuniformly) note any $\deltac$-name of $X$ is also a $\deltamc$-name of $X$.
For given $x\in X$ and $k\in\N$ we will compute a $\delta_{\Delta^0_1}$-name of a neighbourhood $W\ni x$ with
$\diam W\leq 2^{-k}$.  Namely, if $p\in\delta_X^{-1}\{x\}$ we will apply \pref{prn:bilocated} twice to
$d(x,\cdot)$ (in place of $f$) to get some $0<\alpha_0<\alpha_1<2^{-k-1}$ with
\begin{equation*}
\cl{\ball{x}{\alpha_i}}=\clball{x}{\alpha_i} \wedge
\cl{X\setminus\clball{x}{\alpha_i}}=X\setminus\ball{x}{\alpha_i}
\end{equation*}
for each $i$.  Then, in particular, $A:=\clball{x}{\alpha_1}\setminus\ball{x}{\alpha_0}$ is the closure of
$V:=\ball{x}{\alpha_1}\setminus\clball{x}{\alpha_0}$.

\noindent{Proof of claim:}
Clearly
\begin{equation*}
y\in V\iff \alpha_0<d(x,y)<\alpha_1 \implies \alpha_0\leq d(x,y)\leq\alpha_1\iff y\in A,
\end{equation*}
so also $\cl{V}\subseteq A$.  Conversely, if $y\in A$, either $\tau:=d(x,y)\in(\alpha_0,\alpha_1)$ or
$\tau=\alpha_0$ or $\tau=\alpha_1$.  If $\tau=\alpha_0$, use
$\cl{X\setminus\clball{x}{\alpha_0}}=X\setminus\ball{x}{\alpha_0}$ to get some sequence
$(y_j)_{j\in\N}\subseteq X\setminus\clball{x}{\alpha_0}$ convergent to $y$.  For large $j$ we get $y_j\in V$,
so $y\in\cl{V}$.  Case $\tau=\alpha_1$ is similar using $\cl{\ball{x}{\alpha_1}}=\clball{x}{\alpha_1}$.  This completes proof of the claim.

Note Cauchy names of $\alpha_0,\alpha_1$ allow us to compute a $\delta_{\Sigma^0_1}$-name of
$V=d(x,\cdot)^{-1}(\alpha_0,\alpha_1)$ from $p$, hence (using properties of formal inclusion) a $\deltar$-name
of $\cl{V}=A$.  On the other hand, a $\delta_{\Pi^0_1}$-name of $A=d(x,\cdot)^{-1}[\alpha_0,\alpha_1]$ can be
used to compute a $\delta_{\textrm{dist}}^{>}$-name of $A$ (and similarly for $A\union\{x\}$), since (in notation of \cite{BrattkaPresser})
$\deltac\leq\delta_{\mathcal{K}}^{>}$ and $\delta^{>}\leq\delta_{\textrm{dist}}^{>}$, and
$\cap:\Pi^0_1(X)\times\mathcal{K}_{>}(X)\to\mathcal{K}_{>}(X)$ is computable.  Consequently, a name of some
$f\in E(A\union\{x\})$ is available.  Now let $W:=\clball{x}{\alpha_1}\cap f^{-1}\{x\}$.  Since $f(A)$ is
disjoint from $x$, also $W=\ball{x}{\alpha_0}\cap f^{-1}\ball{x}{\alpha_0}$,
and the result follows (since $2\alpha_0\leq 2^{-k}$).
\end{proof}

More formally, one can extract from \pref{thm:retracta} and \pref{prn:retractb} the following equivalence statement: if computable metric space $X$ is $\deltac$-computable then
\begin{equation*}
\dim X=0\iff E \text{ well-defined \& computable }\iff X \text{ eff.~of covering dimension }\leq 0.
\end{equation*}
This uses $\deltadc\equiv\deltac|^{\mathcal{Z}_{\textrm{c}}(X)}$; we leave details to the reader.

\begin{acknowledgements*}
The author is very grateful to the anonymous referees, whose comments have helped in improving the presentation
of this paper.  Many thanks also to Paul Wright, who proofread the paper at a late stage.
\end{acknowledgements*}
\bibliographystyle{plain}
\bibliography{short3}
\end{document}